\documentclass[12pt,letterpaper, draft]{article}

\usepackage[intlimits,sumlimits]{amsmath}
\usepackage{amssymb}
\usepackage{amsthm}

\newcommand{\beq}{\begin{equation}}
\newcommand{\eeq}{\end{equation}}

\newcommand{\myboldfont}{\mathbf}

\newcommand{\eps}{\varepsilon}
\renewcommand{\epsilon}{\varepsilon}
\renewcommand{\kappa}{\varkappa}
\renewcommand{\ge}{\geqslant}
\renewcommand{\le}{\leqslant}
\renewcommand{\leq}{\leqslant}

\renewcommand{\emptyset}{\varnothing}

\makeatletter   
\let\th\@undefined               
\makeatother

\DeclareMathOperator{\th}{th}

\DeclareMathOperator{\leb}{Leb}

\newcommand{\SL}{\ensuremath{\mathrm{SL}}}

\newcommand{\SO}{\ensuremath{\mathrm{SO}}}

\newcommand{\quot}{\ensuremath{\backslash}}

\renewcommand{\S}{\ensuremath{\myboldfont S}}
\newcommand{\R}{\ensuremath{\myboldfont R}}

\newcommand{\Z}{\ensuremath{\myboldfont Z}}

\newcommand{\Q}{\ensuremath{\myboldfont Q}}

\newcommand{\N}{\ensuremath{\myboldfont N}}
\DeclareMathOperator{\EE}{\ensuremath{\myboldfont E}}
\DeclareMathOperator{\PP}{\ensuremath{\myboldfont P}}

\renewcommand{\d}{\ensuremath{\partial}}

\renewcommand{\myboldfont}{\mathbb}

\usepackage[margin=2.5cm, includefoot]{geometry}

\usepackage{bm}
\usepackage{enumerate, mathrsfs, mathtools}

\newtheorem{theorem}{Theorem}[section]
\newtheorem{maintheorem}[theorem]{Main Theorem}
\newtheorem{lemma}[theorem]{Lemma}

\newtheorem{prop}[theorem]{Proposition}
\newtheorem{cor}[theorem]{Corollary}

\newcounter{theremark}
\numberwithin{equation}{section}

\usepackage{graphicx, bbm}

\usepackage[font=small,labelfont=bf]{caption}

\makeatletter

\renewcommand{\section}{\@startsection{section}{0}{0pt}{3.6ex plus 0.2ex minus 0.1ex}{2.3ex plus 0.1ex minus 0.1ex}{\center\normalfont\sc\large}}

\makeatother

\def\NN{{\mathbb N}}
\def\QQ{{\mathbb Q}}
\def\RR{{\mathbb R}}

\def\ZZ{{\mathbb Z}}

\def\veca{{\text{\boldmath$a$}}}

\def\vece{{\text{\boldmath$e$}}}

\def\vecm{{\text{\boldmath$m$}}}

\def\vecq{{\text{\boldmath$q$}}}

\def\vecu{{\text{\boldmath$u$}}}
\def\vecv{{\text{\boldmath$v$}}}

\def\vecx{{\text{\boldmath$x$}}}

\def\vecy{{\text{\boldmath$y$}}}
\def\vecz{{\text{\boldmath$z$}}}

\def\vecbeta{{\text{\boldmath$\beta$}}}

\def\vecxi{{\text{\boldmath$\xi$}}}

\def\bs{\backslash}

\def\vecnull{{\text{\boldmath$0$}}}

\def\scrB{{\mathcal B}}

\def\scrK{{\mathcal K}}
\def\scrL{{\mathcal L}}
\def\scrM{{\mathcal M}}

\def\scrP{{\mathcal P}}

\def\scrX{{\mathcal X}}

\def\curA{{\mathscr A}}
\def\curB{{\mathscr B}}
\def\curC{{\mathscr C}}

\def\e{\mathrm{e}}

\def\sep{\operatorname{sep{}}}
\def\S{\operatorname{S{}}}

\def\SL{\operatorname{SL}}

\def\SO{\operatorname{SO}}

\def\GamG{\Gamma\backslash G}

\def\bs{\backslash}

\def\tmu{{\widetilde\mu}}

\makeatletter
\newcommand{\subalign}[1]{%
  \vcenter{%
    \Let@ \restore@math@cr \default@tag
    \baselineskip\fontdimen10 \scriptfont\tw@
    \advance\baselineskip\fontdimen12 \scriptfont\tw@
    \lineskip\thr@@\fontdimen8 \scriptfont\thr@@
    \lineskiplimit\lineskip
    \ialign{\hfil$\m@th\scriptstyle##$&$\m@th\scriptstyle{}##$\crcr
      #1\crcr
    }%
  }
}
\makeatother

\newcommand{\group}{\ensuremath{G_0}}
\newcommand{\lattice}{\ensuremath{\Gamma_0}}
\newcommand{\affg}{\ensuremath{G}}
\newcommand{\affl}{\ensuremath{\Gamma}}

\newcommand{\affm}{\ensuremath{\mu}}

\newcommand{\bQ}{Q}
\newcommand{\bP}{P}
\newcommand{\1}{\mathbbm{1}}

\def \toweak {\,\,\buildrel{\rm w}\over\longrightarrow\,\,}

\def \todist {\,\,\buildrel{\rm d}\over\longrightarrow\,\,}

\begin{document}
\title{Spherical averages in the space of marked lattices}
\author{Jens Marklof\thanks{School of Mathematics, University of Bristol, Bristol BS8 1TW, U.K.} \and Ilya Vinogradov\thanks{Department of Mathematics, Princeton University, Princeton NJ 08544, USA}}

\date{\today}

\maketitle

\begin{abstract}
A marked lattice is a $d$-dimensional Euclidean lattice, where each lattice point is assigned a mark via a given random field on ${\mathbb Z}^d$. We prove that, if the field is strongly mixing with a faster-than-logarithmic rate, then for {\em every} given lattice and {\em almost every} marking, large spheres become equidistributed in the space of marked lattices.  A key aspect of our study is that the space of marked lattices is not a homogeneous space, but rather a non-trivial fiber bundle over such a space. As an application, we prove that the free path length in a crystal with random defects has a limiting distribution in the Boltzmann-Grad limit. 
\end{abstract}

\section{Introduction}

Consider a Lie group $G$, a non-compact one-parameter subgroup $\Phi^\RR$ and a compact subgroup $K$. Let $\lambda$ be a probability measure on $K$ that is absolutely continuous with respect to Haar measure on $K$. Given a measure-preserving action $G\times X\to X$, $(g,x)\mapsto xg$ on a probability space $(X,\curA,\mu)$, it is natural to ask under which conditions the ``spherical'' average $P_t$ defined by $P_t f:=\int_K f(x_0 k \Phi^t) d\lambda(k)$ converges weakly to $\mu$, or any other probability measure. In general the best one can hope for is convergence for $\mu$-almost all $x_0$. Proofs typically require an additional average over $\Phi^t$, and may be viewed as generalizations of the classic Wiener ergodic theorem; see Nevo's survey \cite{nevo_pointwise_2006} and references therein. If the space $X$ is homogeneous, then the weak convergence of the spherical average $P_t$ can be proved for all $x_0$, with a complete classification of all limit measures, by means of measure rigidity techniques that are based on Ratner's measure classification theorem for subgroups generated by unipotent elements \cite{ratner_raghunathans_1991}. There is now a large body of literature on this topic, see for instance  \cite{shah_limit_1996, Eskin1998,Eskin2005,marklof_pair_correlation_2003,marklof_strombergsson_free_path_length_2010,marklof_strombergsson_free_path_length_2014}.
In some settings, spherical equidistribution may also be deduced directly from the mixing property of $\Phi^\RR$ \cite{eskin_mixing_1993}.  The first example of spherical equidistribution in the non-homogeneous setting for all (and not just almost all) $x_0$ is given in \cite{eskin_marklof_morris_2006}, where the analogue of Ratner's theorem is proved for the moduli space of branched covers of Veech surfaces, which is a fiber bundle over a homogeneous space. A major advance in this direction is the recent work by Eskin and Mirzakhani \cite{eskin_invariant_2013} and Eskin, Mirzakhani  and Mohammadi \cite{eskin_isolation_2013}, who prove a Ratner-like classification of measures in the moduli space of flat surfaces that are invariant under the upper triangular subgroup of $\SL(2,\R)$. This is used to prove convergence of spherical averages in that moduli space, with an additional $t$ average as above, which yields an averaged counting asymptotics for periodic trajectories in general rational billiards.

The goal of the present study is to construct a natural example of a non-homogeneous space (the space of marked Euclidean lattices), which is a fiber bundle over a homogeneous space (the space of Euclidean lattices), and to prove spherical equidistribution for \emph{every} point in the base and \emph{almost every} point in the fiber. 
Our findings complement a theorem of Brettschneider \cite[Theorem~4.7]{brettschneider_uniform_2011}, who proves uniform convergence of Birkhoff averages for fiber bundles with uniquely ergodic base under technical assumptions on the test function and fiber transformation. 

This paper is organized as follows. We introduce the space of lattices in Section \ref{sec:lattice}, then the space of marked lattices in Section \ref{sec:marked}, where the marking is produced by a random field on $\ZZ^d$. The main results of this study, limit theorems for spherical averages in the space of marked lattices, are stated and proved in Section \ref{sec:main0} and \ref{sec:main}. The former deals with convergence on average over the field, the latter with a fixed realization of the random field. Section \ref{sec:defect} applies these results to the setting of defect lattices, where lattice points are either randomly removed, or shifted from their equilibrium position.
Section \ref{sec:lorentz} explains how these findings can be used to calculate the limit distribution for the free paths lengths in the Boltzmann-Grad limit of a Lorentz gas for such scatterer configurations. 

{\bf Acknowledgements.} We thank Alex Eskin, Alex Furman, Amos Nevo, and Andreas Str\"ombergsson for stimulating discussions, and MSRI for its hospitality during the programme ``Geometric and Arithmetic Aspects of Homogeneous Dynamics.'' The research leading to these results has received funding from the European Research Council under the European Union's Seventh Framework Programme (FP/2007-2013) / ERC Grant Agreement n. 291147.

\section{Spherical averages in the space of lattices\label{sec:lattice}}

Let $\group=\SL(d,\R)$ and let $\lattice=\SL(d,\Z)$. We represent elements in $\R^d$ as row vectors, and define a natural action of $\group$ on $\R^d$ by right matrix multiplication. The map
\begin{equation}
\lattice M \mapsto \Z^d M
\end{equation}
gives a one-to-one correspondence between the homogeneous space $\lattice\quot\group$ and the space of Euclidean lattices in $\R^d$ of covolume one. The Haar measure $\mu_0$ on $G_0$ is normalized, so that it projects to a probability measure on $\Gamma_0\bs G_0$ which will also be denoted by $\mu_0$.

Let $\affg = \group\ltimes \R^d$ be the semidirect product with multiplication law
\begin{equation}
(M,\vecxi)(M',\vecxi')=(MM',\vecxi M'+\vecxi')
\end{equation}
where $\vecxi,\vecxi'$ are viewed as row vectors.
The group $\affg$ is a bundle over \group\ with fiber $\R^d$.  The subgroup $\affl = \lattice \ltimes \Z^d$ is a lattice in $G$. The Haar measure on $G$ is $\mu=\mu_0\times\leb_{\RR^d}$. It induces a probability measure on $\Gamma\bs G$ which will also be denoted by $\mu$.
The groups \group\ and \affg\ act on $\R^d$ by linear and affine transformations, respectively,  which are given by 
\begin{align}
 \R^d & \curvearrowleft \group & \R^d & \curvearrowleft \affg\\
 (\vecv,g)& \mapsto \vecv g & (\vecv,(M,\vecxi))& \mapsto \vecv M +\vecxi =: \vecv(M,\vecxi),
\intertext{and}
\Z^d & \curvearrowleft \lattice & \Z^d & \curvearrowleft \affl\\
 (\vecv,g)& \mapsto \vecv g & (\vecv ,(M,\vecxi))& \mapsto \vecv M +\vecxi = \vecv (M,\vecxi),\label{eq:discrete_action_eq}
\end{align}
where concatenation denotes matrix multiplication. We embed $\group \hookrightarrow \affg$ by $M\mapsto(M,\vecnull)$, and identify $G_0$ with its image under this embedding.

As in the linear case, the map
\begin{equation}
\affl g \mapsto \Z^d g
\end{equation}
gives a one-to-one correspondence between the homogeneous space $\affl\quot\affg$ and the space of affine Euclidean lattices in $\R^d$ of covolume one.

We will need other subgroups of \affg\ in addition to \group. For $\vecxi\in\R^d$, put
\begin{align}
 \affg_\vecxi=
 \begin{cases}
  \affg & \text{ if }\vecxi \in \R^d\setminus \Q^d,\\
  (\1,\vecxi) \group (\1,\vecxi)^{-1} & \text{ if } \vecxi\in\Q^d.
 \end{cases}
\end{align}
The subgroup $\affl_\vecxi = \affl\cap \affg_\vecxi$ is a lattice in $\affg_\vecxi$. We denote by $\mu_{\vecxi}$ be the Haar measure on $\affg_\vecxi$, normalized so that a fundamental domain of the $\affl_\vecxi$-action in $G_\vecxi$ has measure 1. We denote the induced probability measure on $\affl_\vecxi\quot \affg_\vecxi$ also by $\mu_{\vecxi}$. When $\vecxi\notin \Q^d$, we put $\affl_\vecxi = \affl$ and $\mu_{\vecxi} = \mu$. Thus, when $\vecxi\in\Q^d$, $\mu_{\vecxi}$ can be identified with a singular measure on $\affl\quot \affg$ supported on the closed subspace $\affl\quot\affl\affg_\vecxi\simeq \affl_\vecxi\quot \affg_\vecxi$ of $X$. Define the translate 
\begin{equation}
X_\vecxi:=\affl\quot\affl\affg_\vecxi (\1,\vecxi),
\end{equation}
again a closed subspace of $X$, which we equip with the subspace topology. Note that $X_\vecxi=\overline{\affl\quot\affl (\1,\vecxi) \group}$. The measure $\tmu_\vecxi$ on $X_\vecxi$ defined as the translate of $\affm_{\vecxi}$,
\begin{equation}
\tmu_\vecxi A = \mu_\vecxi(A(\1,\vecxi)^{-1})
\end{equation}
for any Borel set $A\subset X_\vecxi$.
For $t\in\R$ and $\vecu \in U\subset \R^{d-1}$, with $U$ open and bounded, define the matrices 
\begin{align}
\Phi^t &=\begin{pmatrix}e^{-(d-1)t} & \vecnull\\\vecnull^T &\1_{d-1}e^t\end{pmatrix}, & R(\vecu) = \exp \begin{pmatrix}0 &\vecu \\-\vecu ^T & \1_{d-1} 0\end{pmatrix}.
\end{align}
The map $U\to \S_1^{d-1}$, $\vecu\mapsto \vece_1 R(\vecu)^{-1}$, where $\vece_1$ is the first standard basis vector, is a diffeomorphism onto its image if $U$ is sufficiently small; cf.\ Remark 5.5 in \cite{marklof_strombergsson_free_path_length_2010}. 
More generally, we can take any smooth map $E\colon U\to \SO(d)$ such that 
\begin{equation}\label{Etilde}
\widetilde E: U\to \widetilde E(U)\subset \S_1^{d-1}, \vecu \mapsto \vece_1 E(\vecu)^{-1},
\end{equation}
is invertible and the inverse is uniformly Lipschitz. We will furthermore assume in the following that the closure of $\widetilde E(U)$ is contained in the hemisphere $\{ \vecv\in \S_1^{d-1} : \vece_1\cdot\vecv>0\}$, and that $\leb(\d U) = 0$.

For a given absolutely continuous probability measure $\lambda$ on $U$, $t\ge 0$ and $\vecxi\in\R^d$, $M \in\group$, let $\bP_t=\bP_t^{(\lambda,M,\vecxi)}$ be the Borel probability measure on $X_\vecxi$ defined by 
\begin{align}\label{eq:measure_def}
\bP_t f = \int f d\bP_t = \int_{\vecu \in U} f(\affl (\1,\vecxi)M E(\vecu) \Phi^t) \lambda(d\vecu) 
\end{align}
for any bounded continuous $f\colon X_\vecxi \to \R$. Note that the restriction to maps $E\colon U\to\SO(d)$ is purely for technical convenience. There is no loss of generality, since $M$ is arbitrary and the maps $U\mapsto\S_1^{d-1}$, $\vecu\mapsto \vece_1 E(\vecu)^{-1} M^{-1}$ cover the sphere for finitely many choices of $M\in\SO(d)$; cf.\ \cite{marklof_strombergsson_free_path_length_2010}.

\begin{theorem}[{\cite[Sec.~5]{marklof_strombergsson_free_path_length_2010}}]\label{th:ms_full}
For $t\to\infty$,
\begin{equation}
P_t \toweak \tmu_\vecxi.
\end{equation}
\end{theorem}

Recall the above weak convergence means that for every bounded continuous $f\colon X_\vecxi \to \R$, $\lim_{t\to\infty} P_t f = \tmu_\vecxi f$.

We will now show how the space of lattices can be viewed as a subspace of the space of point processes in $\RR^d$. The extension of this to marked point processes will be a key element in this paper.

Let $\scrM(\RR^d)$ be the space of locally finite Borel measures on $\RR^n$, equipped with the vague topology. The vague topology is the smallest topology such that the function 
\begin{equation}\widehat f:\scrM(\RR^d)\to\RR,\qquad \mu\mapsto \mu f\end{equation}
is continuous for every $f\in C_c(\RR^d)$ (the space of continuous functions $\RR^d\to\RR$ with compact support). The space $\scrM(\RR^d)$ is Polish 
in this topology \cite[Theorem A 2.3]{Kallenberg02}.
We embed the space of affine lattices in $\scrM(\RR^d)$ by the map
\begin{equation}
\iota\colon X\to \scrM(\RR^d),\qquad x \mapsto \sum_{\vecy\in\ZZ^d x} \delta_\vecy  .
\end{equation}
For technical reasons (which will become clear in Corollary \ref{cor:ms_special}) we will need to treat the space of lattices $X_0$ slightly differently; define
\begin{equation} \label{iota0}
\iota_0\colon X_0\to \scrM(\RR^d),\qquad x \mapsto \sum_{\vecy\in\ZZ^d x\setminus\{\vecnull\}} \delta_\vecy  .
\end{equation}

\begin{prop}\label{topprop1}
The maps $\iota$ and $\iota_0$ are topological embeddings.
\end{prop}

\begin{proof} 
To establish the continuity of $\iota$, we need to show that, for every $f\in C_c(\RR^d)$, $x_j\to x$ in $X$ implies $\iota(x_j)f \to \iota(x) f$. By the $\Gamma$-equivariance of $\iota$, it is sufficient to show that $g_j\to g$ in $G$ implies
\begin{equation} \sum_{\vecy\in\ZZ^d g_j} f(\vecy) \to  \sum_{\vecy\in\ZZ^d g} f(\vecy).\end{equation}
Let $A$ be the compact support of $f$. Since $g_j\to g$, the closure of $A'=\cup_j (Ag_j^{-1})$ is compact. Hence $\ZZ^d\cap A'$ is finite. For $\veca\in\ZZ^d\setminus A'$ we have $f(\veca g_j)=f(\veca g)=0$, and for the finitely many $\veca\in\ZZ^d\cap A'$ we have $f(\veca g_j)\to f(\veca g)$.

The map $\iota$ is injective, since the lattice $\ZZ^d x$ uniquely determines $x\in\GamG$. Let $\tilde\iota\colon X \to \iota(X)$, $x\mapsto\iota(x)$.
To establish the continuity of $\tilde\iota^{-1}$, we need to show that $\iota(x_j)f \to \iota(x) f$ for every $f\in C_c(\RR^d)$ implies $x_j\to x$ in $X$. Fix $g=(M,\vecxi)\in\Gamma x$.
Then $\vece_1 M,\ldots,\vece_n M$ forms a basis of $\ZZ^d M$, where $\vece_k$ are the standards basis vectors of $\ZZ^d$. Set $\vece_0=\vecnull$ and define for $k=0,1,\ldots,n$,
\begin{equation}
f_{k,\delta}(\vecy)=
\begin{cases} 
1 & \text{if $\| \vecy - \vece_k g\| < \frac{\delta}{2}$}\\
2- \frac{2}{\delta} \| \vecy - \vece_k g\| & \text{if $\frac{\delta}{2} \le \| \vecy - \vece_k g\| <  \delta$}\\
0 & \text{if $\| \vecy - \vece_k g\| \ge \delta$.}
\end{cases}
\end{equation}
Note that $f_{k,\delta} \in C_c(\RR^d)$. By the discreteness of $\ZZ^d g$, there is $\delta_0>0$ such that for all $\delta\in(0,\delta_0]$, all $k=0,1,\ldots,n$,
\begin{equation} \iota(x) f_{k,\delta} = \sum_{\vecy\in\ZZ^d g} f_{k,\delta}(\vecy) = f_{k,\delta}(\vece_k g) = 1 .\end{equation}
Since by assumption $\iota(x_j)f_{k,\delta} \to \iota(x) f_{k,\delta}=1$, given $\delta>0$, there is $j_0\in\NN$ such that for every $j\ge j_0$ and for every $k$, there is at least one element in $\ZZ^d x_j$ within distance $\delta$ to $\vece_k g$. Call this element $\vecy_k^{(j)}$. Then 
\begin{equation}\label{222}
\vecy_k^{(j)}\to \vece_k g= \vece_k M +\vecxi \quad \text{for every $k=0,1,\ldots,d$,} 
\end{equation}
and therefore
\begin{equation}\label{2222}
\vecy_k^{(j)}-\vecy_0^{(j)} \to \vece_k g -\vece_0 g = \vece_k M \quad \text{for every $k=1,\ldots,d$.} 
\end{equation}
Because of this and the fact that the lattices $\ZZ^d x_j-\vecy_0^{(j)}$ and $\ZZ^d x-\vecxi$ both have covolume one, 
the vectors $\vecy_1^{(j)}-\vecy_0^{(j)},\ldots,\vecy_n^{(j)}-\vecy_0^{(j)}$ form a basis of $\ZZ^d x_j-\vecy_0^{(j)}$ (for all sufficiently large $j$). Then
\begin{equation} M_j := \begin{pmatrix} \vecy_1^{(j)}-\vecy_0^{(j)}\\ \vdots \\ \vecy_n^{(j)}-\vecy_0^{(j)} \end{pmatrix} \in G_0,\qquad \vecxi_j:=\vecy_0^{(j)}  .\end{equation}
Now \eqref{222} implies $g_j=(M_j,\vecxi_j)\to g=(M,\vecxi)$ and thus $x_j\to x$.

The proof for $\iota_0$ is almost identical.
\end{proof}

Every random element $\zeta$ in $X$ defines a point process $\Theta=\iota(\zeta)$ in $\scrM(\RR^d)$. Let $\zeta_t$ be the random element distributed according to $P_t$, and $\zeta$ according to $\tmu_\vecxi$, with $\vecxi\notin\ZZ^d$. Theorem \ref{th:ms_full} can then be rephrased as $\zeta_t \todist \zeta$. In view of Proposition \ref{topprop1} and the continuous mapping theorem \cite[Theorem 4.27]{Kallenberg02}, this is equivalent to the following convergence in distribution for the point processes $\Theta_t=\iota(\zeta_t)$ and $\Theta=\iota(\zeta)$ in the case $\vecxi\notin\ZZ^d$, $\Theta_{0,t}=\iota_0(\zeta_t)$ and $\Theta_0=\iota_0(\zeta)$ for $\vecxi\in\ZZ^d$. To simplify notation we suppress the dependence on $\vecxi$; $\Theta$ depends on the choice of $\vecxi\in\RR^d\setminus\ZZ^d$.

\begin{theorem}\label{thm:ppoint}
For $t\to\infty$,
\begin{equation}
\Theta_t \todist \Theta \quad (\vecxi\notin\ZZ^d), \qquad \Theta_{0,t} \todist \Theta_0 \quad (\vecxi\in\ZZ^d).
\end{equation}
\end{theorem}

We now turn to the finite-dimensional distribution of the above point processes, cf.~{\cite[Sec.~5]{marklof_strombergsson_free_path_length_2010}}.

\begin{cor}\label{cor:ms_special}
Let $n\in\NN$ and $A_1,\ldots,A_n\subset\RR^d$ bounded Borel sets with $\leb (\d A_i) = 0$ for all $i$. Then, for $t\to\infty$,
\begin{equation}
(\Theta_t A_1,\ldots,\Theta_t A_n) \todist (\Theta A_1,\ldots,\Theta A_n)  \quad (\vecxi\notin\ZZ^d),
\end{equation}
\begin{equation}
(\Theta_{0,t} A_1,\ldots,\Theta_{0,t} A_n) \todist (\Theta_0 A_1,\ldots,\Theta_0 A_n) \quad (\vecxi\in\ZZ^d).
\end{equation}
\end{cor}

In view of  \cite[Theorem 16.16]{Kallenberg02}, the main ingredient in the derivation of Corollary \ref{cor:ms_special} from Theorem \ref{thm:ppoint} is to show that $\leb (\d A_i) = 0$ implies that $\Theta \d A_i=0$ almost surely, and $\Theta_0 \d A_i=0$ almost surely. This follows from Siegel's integral formula \cite{siegel_mean_1945, veech_siegel_1998}, which says that $\EE\Theta_0 B=\leb (B)$, $\EE \Theta B=\leb (B)$ for every $B\in\curB(\RR^d)$ (note that this identity is straightforward for $\vecxi\notin\QQ^d$, since it follows directly from the translation invariance of $\Theta$). Note that $\Theta \d A_i=0$ fails for $\vecxi\in\ZZ^d$ if $\vecnull\in\d A_i$. This is the reason for removing $\vecnull$ in the definition \eqref{iota0} of $\iota_0$. But Siegel's formula implies $\leb (\d A_i) = 0$ {\em if and only if} $\Theta_0 \d A_i=0$ (resp.~$\Theta \d A_i=0$) almost surely. Therefore the statement of Corollary  \ref{cor:ms_special} is in fact equivalent to Theorem  \ref{thm:ppoint} via \cite[Theorem 16.16]{Kallenberg02}. We will exploit the analogue in the treatment of marked lattices.

The following lemmas will be useful below.

\begin{lemma}\label{lem:onepoint}
For $A\in\curB(\RR^d)$, 
 \begin{equation}
  \PP(\Theta A\ge 1) \le \leb A  \quad (\vecxi\notin\ZZ^d), \qquad \PP(\Theta_0 A\ge 1) \le \leb A .
 \end{equation}

\end{lemma}

\begin{proof}
This follows from Chebyshev's inequality followed by Siegel's formula. 
\end{proof}

\begin{lemma}\label{lem:decay}
For $A\in\curB(\RR^d)$ and $L \in\Z_{\ge 0}$, 
\begin{equation}
\PP(\Theta A\ge L)
\ll_{A,\vecxi} \begin{cases} (1+L)^{-d-1} & (\vecxi\notin\Q^d) ,\\ 
(1+L)^{-d} &(\vecxi\in\Q^d\setminus\ZZ^d),
\end{cases}
\end{equation}
and
\begin{equation}
\PP(\Theta_0 A\ge L)
\ll_{A} 
(1+L)^{-d} .
\end{equation}
\end{lemma}

\begin{proof}
 See \cite[Theorems~4.3, 4.5]{marklof_2000}. 
\end{proof}

\section{Marked lattices and marked point processes}\label{sec:marked}

We will now extend the discussion in the previous section to the space of {\em marked} lattices, which is defined as a certain fiber bundle over the space of lattices. The key point is now to identify this space with a {\em marked} point process.

Each map $\omega\colon \ZZ^d \to Y$, where $Y$ is the {\em set of marks}, produces a marking of the affine lattice $\ZZ^d g$ with $g\in G$: the point $\vecy\in\ZZ^d g$ has mark $\omega(\vecy g^{-1})$. A $Y$-marked affine lattice is thus the point set
\begin{equation}\label{MPS}\{(\vecm g,\omega(\vecm))\mid\vecm\in\Z^d\}\end{equation}
in $\R^d\times Y$, and 
can be parametrized by the pair $(g,\omega)\in G \times \Omega$, where $\Omega=\{ \omega: \ZZ^d \to Y\}$ is the set of all possible markings. Note that, for $\gamma\in\Gamma$, the point $\vecy=\vecm \gamma g\in\ZZ^d g$ has mark \begin{equation}\omega(\vecy g^{-1})=\omega(\vecm\gamma)=\omega_\gamma(\vecy(\gamma g)^{-1}),\end{equation} 
where $\omega_\gamma(\vecm):=\omega(\vecm\gamma)$. Hence $(g,\omega)$ and $(\gamma g,\omega_\gamma)$ yield the same marked affine lattice.
This motivates the definition of the left action of $\Gamma$ on $G \times \Omega$ by $\gamma(g,\omega):=(\gamma g, \omega_\gamma)$. We define a right action of $G$ on $G \times \Omega$ by $(g,\omega) g':=(g g', \omega)$. In analogy with the homogeneous space setting we define
\begin{equation}
\scrX:= \Gamma\bs (G\times\Omega) 
\end{equation}
and
\begin{equation}
\scrX_\vecxi:= \Gamma\bs \Gamma(G_\vecxi\times\Omega)(\1,\vecxi) .
\end{equation} 
For the case $\vecxi\in\Z^d$, we have
\begin{equation}
\scrX_\vecxi = \Gamma\bs \Gamma((\1,\vecxi) G_0\times\Omega) 
=  \Gamma\bs \Gamma(G_0\times (\1,-\vecxi) \Omega) =\Gamma\bs \Gamma(G_0\times \Omega) =\scrX_\vecnull .
\end{equation} 
which in turn can be identified with the space $\Gamma_0\bs (G_0\times \Omega)$ via the map $\Gamma ((M,\vecnull),\omega)\mapsto \Gamma_0(M,\omega)$. Note that the point $\Gamma((\1,\vecxi) M,\omega)=\Gamma(M, \omega_{(\1,-\vecxi)})$ is mapped to $\Gamma_0(M,\omega_{(\1,-\vecxi)})$ under this identification. 

\begin{lemma}\label{lemmi1}
The map 
\begin{equation}
\Gamma(g,\omega)\mapsto \{(\vecm g,\omega(\vecm))\mid\vecm\in\Z^d\}
\end{equation}
yields a one-to-one correspondence between $\scrX$ and $Y$-marked affine lattices of covolume one. 
\end{lemma}

\begin{proof}
For $(g,\omega)$ and $(g',\omega')$ to yield the same marked lattice, it is necessary that $\ZZ^d g=\ZZ^d g'$. Hence $g'=\gamma g$ for some $\gamma\in\Gamma$. But this implies $\omega'=\omega_\gamma$ and hence $(g',\omega')=\gamma(g,\omega)$.
\end{proof}

We now extend the above correspondences to the topological setting. Let $Y$ be a topological space, and endow the space of all markings $\Omega=Y^{\ZZ^d}$ with the product topology. Define the topology of $G\times \Omega$ by the product topology, and on $\scrX$, $\scrX_\vecxi$ by the quotient and subspace topology, respectively.
If $Y$ is locally compact second countable Hausdorff (lcscH), then $\RR^d\times Y$ is lcscH. Consider the measurable space $(Y,\curB(Y))$ with Borel $\sigma$-algebra $\curB(Y)$, and define $\scrM(\RR^d\times Y)$ as the space of $\sigma$-finite Borel measures on $\RR^d\times Y$ equipped with the vague topology. Under these assumptions, $\scrM(\RR^d\times Y)$ is Polish \cite[Theorem A 2.3]{Kallenberg02}.

Set
\begin{equation}
\kappa\colon \scrX\to \scrM(\RR^d\times Y),\qquad \Gamma(g,\omega) \mapsto \sum_{\vecy\in\ZZ^d g} \delta_{(\vecy,\omega(\vecy g^{-1}))}  
\end{equation}
and 
\begin{equation}
\kappa_0\colon \scrX_\vecnull \to \scrM(Y)\times \scrM(\RR^d\times Y),\qquad \Gamma(g,\omega) \mapsto \bigg(\delta_{\omega(\vecnull g^{-1})},\sum_{\vecy\in\ZZ^d g\setminus\{\vecnull\}} \delta_{(\vecy,\omega(\vecy g^{-1}))} \bigg).
\end{equation}
These maps are well-defined and injective by Lemma \ref{lemmi1}. Note that $\kappa_0$ maps the point $\Gamma((\1,\vecxi) M,\omega)=\Gamma(M, \omega_{(\1,-\vecxi)})$ with $\vecxi\in\ZZ^d$, $M\in G_0$ to 
\begin{equation}
\bigg(\delta_{\omega(-\vecxi)},\sum_{\vecm\in\ZZ^d\setminus\{\vecnull\}} \delta_{(\vecm M,\omega(\vecm-\vecxi )} \bigg) .
\end{equation}

\begin{prop}\label{topprop2}
The maps $\kappa$ and $\kappa_0$ are topological embeddings.
\end{prop}

\begin{proof}
To prove continuity of $\kappa$, we need to show that $(g_j,\omega_j)\to (g,\omega)$ in $G\times \Omega$ implies
\begin{equation} \sum_{\vecy\in\ZZ^d g_j} f(\vecy,\omega_j(\vecy g_j^{-1})) \to  \sum_{\vecy\in\ZZ^d g} f(\vecy,\omega(\vecy g^{-1})) \label{eq:kappa_conv}
\end{equation}
for every $f\in C_c(\RR^d\times Y)$.
As in the proof of Proposition \ref{topprop1}, the compact support of $f$ reduces the problem to showing that
$f(\veca g_j,\omega_j(\veca))\to f(\veca g,\omega(\veca))$ for finitely many $\veca\in\ZZ^d$. The latter follows from the continuity of $f$.

Let $\tilde\kappa\colon \scrX \to \kappa(\scrX)$, $x\mapsto\kappa(x)$.
To establish the continuity of $\tilde\kappa^{-1}$, we need to show that $\kappa(x_j)f \to \kappa(x) f$ for every $f\in C_c(\RR^d\times Y)$ implies $x_j\to x$ in $\scrX$. We already know from Proposition \ref{topprop1} that $g_j\to g$.  Fix $\vecm\in\ZZ^d$, and define for $g\in G$, $h\in C_c(Y)$,
\begin{equation}
f_{h,\delta}(\vecy,t)=
\begin{cases} 
h(t) & \text{if $\| \vecy - \vecm g\| < \frac{\delta}{2}$}\\
(2- \frac{2}{\delta} \| \vecy - \vecm g\|) h(t) & \text{if $\frac{\delta}{2} \le \| \vecy - \vecm g\| <  \delta$}\\
0 & \text{if $\| \vecy - \vecm g\| \ge \delta$.}
\end{cases}\label{eq:smooth_ind}
\end{equation}
Note that $f_{h,\delta} \in C_c(\RR^d\times Y)$. By the discreteness of $\ZZ^d g$, there is $\delta_0>0$ such that for all $\delta\in(0,\delta_0]$,
\begin{equation} \kappa(x) f_{h,\delta} = \sum_{\vecy\in\ZZ^d g} f_{h,\delta}(\vecy,\omega(\vecy g^{-1})) = f_{h,\delta}(\vecm g,\omega(\vecm)) = h(\omega(\vecm)) .\end{equation}
Since $g_j\to g$, for given $\delta>0$, there is $j_0\in\NN$ such that for  all $j\ge j_0$, 
\begin{equation}  \kappa(x_j)  f_{h,\delta} = \sum_{\vecy\in\ZZ^d g_j} f_{h,\delta}(\vecy,\omega_j(\vecy g_j^{-1})) = f_{h,\delta}(\vecm g_j,\omega_j(\vecm)) = h(\omega_j(\vecm)) .
\end{equation}
Now $ \kappa(x_j) f \to \kappa(x) f$ for every $f\in C_c(\RR^d\times Y)$ implies $h(\omega_j(\vecm))\to h(\omega(\vecm))$ for every fixed $h\in C_c(Y)$ and $\vecm\in\ZZ^d$. That is, $\delta_{\omega_j(\vecm)}\to
\delta_{\omega(\vecm)}$ in $\scrM(Y)$. Any open set $B$ containing $\omega(\vecm)$ satisfies  $\delta_{\omega(\vecm)} B=1$ and $\delta_{\omega(\vecm)}\partial B=0$. Thus $\delta_{\omega_j(\vecm)}B\to 1$ for any open set $B$ containing $\omega(\vecm)$ \cite[Theorem A 2.3]{Kallenberg02}, and therefore $\omega_j(\vecm)\to\omega(\vecm)$. This implies $\omega_j\to\omega$ in the product topology of $\Omega$.

The proof for $\kappa_0$ is similar to the above, with the following modifications. Eq.~\eqref{eq:kappa_conv} is replaced by 
\begin{align}
 \bigg(f_1(\omega_j(\vecnull {g_j}^{-1})),  \sum_{\vecy\in\ZZ^d g_j\setminus\{\vecnull\}} f_2(\vecy,\omega_j(\vecy g_j^{-1})) \bigg)
 \to  
 \bigg(f_1(\omega(\vecnull g^{-1})), \sum_{\vecy\in\ZZ^d g \setminus\{\vecnull\}} f_2(\vecy,\omega(\vecy g^{-1})) \bigg),
\end{align}
for every $(f_1, f_2) \in C_c(Y)\times C_c(\R^d\times Y),$ which is seen to hold as in the argument for $\kappa$. 

Let $\tilde \kappa_0 \colon \scrX \to \kappa_0(\scrX)$, $x\mapsto \kappa_0 (x)$. To establish continuity of $\tilde \kappa_0^{-1}$, it is enough to show that 
\begin{align}
 [\kappa_0(x_j)(f_1, f_2) \to \kappa_0(x)(f_1, f_2), \text{ for all }(f_1,f_2)\in C_c(Y)\times C_c(\R^2\times Y)] \Rightarrow [x_j\to x].
\end{align}
We know from Proposition \ref{topprop1} that $g_j\to g$, and need to show that $\omega_j\to \omega$. For any $\vecm\in\Z^d\setminus \{\vecnull\}$, define $f_{h_2,\delta}$ as in \eqref{eq:smooth_ind}. 
It follows from discreteness of $\Z^d g$ that there is $\delta_0>0$ such that for all $\delta \in (0,\delta_0]$, 
\begin{align} \kappa_0(x) (h_1, f_{h_2,\delta}) & =
\bigg(h_1(\vecnull), \sum_{\vecy\in\ZZ^d g\setminus\{\vecnull\}} f_{h_2,\delta}(\vecy,\omega(\vecy g^{-1})) \bigg)\\
&
= \left( h_1(\omega(\vecnull)) , f_{h_2,\delta}(\vecm g,\omega(\vecm)) \right) \\
&=\left(h_1(\omega(\vecnull) ), h_2(\omega(\vecm))\right) .
\end{align}
Since $g_j\to g$, we have
\begin{align}
 \kappa_0(x_j) (h_1, f_{h_2,\delta}) =\left(h_1(\omega_j(\vecnull) ), h_2(\omega_j(\vecm))\right) 
\end{align}
for all $j\ge j_0$. Thus, we have 
\begin{align}
 \left(h_1(\omega_j(\vecnull) ), h_2(\omega_j(\vecm))\right)  \to  \left(h_1(\omega(\vecnull) ), h_2(\omega(\vecm))\right)
\end{align}
for all bounded continuous functions $h_1$, $h_2$ and $\vecm\in\Z^d\setminus \{\vecnull\}$, inasmuch as 
$ (\delta_{\omega_j(\vecnull)}, \delta_{\omega_j(\vecm)}) \to  (\delta_{\omega(\vecnull)}, \delta_{\omega(\vecm)}) $
in $\scrM(Y) \times \scrM(\R^d\times Y)$ for each $\vecm\in\Z^d\setminus \{\vecnull\}$. This implies that $\omega_j\to\omega$ in the product topology on $\Omega$, as needed. 

\end{proof}

To define probability measures on the space $\scrX_\vecxi$ of affine marked lattices, let us fix a random field $\eta:\ZZ^d\to Y$, $\vecm\mapsto\eta(\vecm)$ defined by the probability measure $\nu$ on the measurable space $(\Omega,\curB)$ (where $\curB=\curB(\Omega)$ is the Borel $\sigma$-algebra on $\Omega=Y^{\ZZ^d}$ with respect to the product topology) via
\begin{equation}\label{eq:field_def}
\PP(\eta(\vecm_1)\in A_1,\ldots, \eta(\vecm_k)\in A_k ) =\nu \{ \omega\in\Omega \mid 
\omega(\vecm_1)\in A_1,\ldots, \omega(\vecm_k)\in A_k \}
\end{equation}
for all $A_1,\ldots,A_k\in\curB(Y)$. In other words, $\eta$ is a random element in $\Omega$ distributed according to $\nu$.

We define the {\em mixing coefficient of order $k$} of the random field $\eta$ by
\begin{equation}
\vartheta_k(s) = \sup\big\{ \vartheta_k(\vecm_1,\ldots,\vecm_k) \big| \vecm_1,\ldots,\vecm_k\in\ZZ^d, \; \|\vecm_i-\vecm_j\|\ge s \text{ if } i\neq j \big\} ,
\end{equation}
where $\|\cdot\|$ is the Euclidean norm, and
\begin{multline}
\vartheta_k(\vecm_1,\ldots,\vecm_k) := \sup_{A_1,\dots,A_k\in \mathscr B(Y)} 
 \big| \PP(\eta(\vecm_1)\in A_1,\ldots, \eta(\vecm_k)\in A_k) \\ - \PP(\eta(\vecm_1)\in A_1) \cdots \PP(\eta(\vecm_k)\in A_k) \big| .
\end{multline}
We say {\em $\eta$ is mixing of order $k$  if}
\begin{equation}
\lim_{s\to\infty} \vartheta_k(s) =0 ,
\end{equation}
and {\em mixing of all orders} if it is mixing of order $k$ for all $k\in\NN$. Note that mixing of order two need not imply mixing of order three; cf.\ Ledrappier's ``three dots'' example \cite{ledrappier_champ_1978,einsiedler_ward_2011}. 

Given $\vecxi\in\RR^d$ and a probability measure $\rho$ on $Y$, we also define  
\begin{equation}\label{mix00}
\beta_{\vecxi}(s)= \sup_{\|\vecm+\vecxi\|\ge s} \sup_{A\in\curB(Y)} \big| \PP(\eta(\vecm)\in A) - \rho(A)\big| .
\end{equation}
If $\lim_{s\to\infty}\beta_{\vecnull}(s)=0$, we say {\em $\eta$ has asymptotic distribution $\rho$}. The presence of $\vecxi$ in \eqref{mix00} is purely for notational convenience further on.

The above mixing conditions will be sufficient for the results in Section \ref{sec:main0}. We will need the following stronger variant for our main results in Section \ref{sec:main}.

Given a non-empty subset $J \subset\ZZ^d$ and a map $a: J\to Y$, we define the {\em cylinder set}
\begin{equation}
\Omega_a = \{ \omega\in\Omega  \mid \omega(\vecm)=a(\vecm) \; \forall \vecm\in J \}.
\end{equation}
The subalgebra generated by all cylinder sets $\Omega_a$ for a given $J$ is denoted by $\curB_J$. The separation of two non-empty subsets $J_1,J_2\subset\ZZ^d$ is defined as
\begin{equation}
\sep(J_1,J_2) = \min \big\{ \|\vecm_1 -\vecm_2 \| : \vecm_1\in J_1,\; \vecm_2\in J_2 \big\}  .
\end{equation}
We define the \emph{strong-mixing coefficient} of the random field $\eta$  by
\begin{equation}
\alpha(s) = \sup\big\{ \alpha(J_1,J_2)  \big| J_1,J_2\subset\ZZ^d \text{ non-empty,} \; \sep(J_1,J_2)\ge s \big\} ,
\end{equation}
where
\begin{equation}
\alpha(J_1,J_2) := \sup \big\{ \big| \nu \big( A_1\cap A_2 \big) - \nu(A_1) \nu(A_2) \big| : A_1\in\curB_{J_1},\; A_2\in\curB_{J_2} \big\} .
\end{equation}
We say {\em $\eta$ is strongly mixing if}
\begin{equation}\label{mix0}
\lim_{s\to\infty} \alpha(s) =0 .
\end{equation}
Note that, in the case of singleton sets, we have
\begin{equation}
\alpha(\{\vecm_1\},\{\vecm_2\} ) = \vartheta_2(\vecm_1,\vecm_2).
\end{equation}
Thus strong mixing implies mixing of order two. In fact, strong mixing implies mixing of any order. This follows from the following observation.
For $k\ge 2$, put
\begin{equation}
\alpha_k(s) := \sup\big\{ \alpha_k(J_1,\ldots,J_k)  \big| \text{$J_1,\ldots, J_k\subset\ZZ^d$ non-empty, $\sep(J_i,J_j)\ge s$ for $i\neq j$} \big\} ,
\end{equation}
where
\begin{equation}
\alpha_k(J_1,\ldots,J_k) := \sup \bigg\{ \bigg| \nu \big( A_1\cap\cdots\cap A_k \big) - \prod_{i=1}^k \nu(A_i) \bigg| : A_i\in\curB_{J_i} \bigg\} .
\end{equation}

\begin{lemma}
If $\alpha(s)\to 0$ then $\alpha_k(s)\to 0$ and $\vartheta_k(s)\to 0$ for all $k\ge 2$.
\end{lemma}

\begin{proof}
Note that $\alpha_k(s)\to 0$ implies $\vartheta_k(s)\to 0$, since 
\begin{equation}
\alpha_k(\{\vecm_1\},\ldots,\{\vecm_k\} ) = \vartheta_k(\vecm_1,\ldots,\vecm_k).
\end{equation}
We will show that $\alpha_k(s)\to 0$ implies $\alpha_{k+1}(s)\to 0$. The claim then follows by induction on $k$.
We have $\curB_{J_i}\subset \curB_J$ for $J:=J_1\cup\cdots\cup J_k$ and therefore
\begin{equation}
A:=A_1\cap\cdots \cap A_k \in \curB_J.
\end{equation}
This in turn implies
\begin{align}
\alpha_{k+1}(J_1,\ldots,J_{k+1}) 
&=  \sup \bigg\{ \bigg| \nu \big( A\cap A_{k+1} \big) - \prod_{i=1}^{k+1} \nu(A_i) \bigg| : A_i\in\curB_{J_i} \bigg\}  \\
&\le  \sup \bigg\{ \bigg| \nu \big( A\cap A_{k+1} \big) - \nu(A) \nu(A_{k+1}) \bigg| : A\in\curB_J,\;  A_{k+1}\in\curB_{J_{k+1}} \bigg\}  \\
& \quad +  \sup \bigg\{ \bigg| \nu(A) \nu(A_{k+1}) - \prod_{i=1}^{k+1} \nu(A_i) \bigg| : A_i\in\curB_{J_i}  \bigg\}  .
\end{align}
The first term equals $\alpha(J,J_{k+1})$, and the second satisfies
\begin{multline}
\sup \bigg\{ \bigg| \nu(A) \nu(A_{k+1}) - \prod_{i=1}^{k+1} \nu(A_i) \bigg| : A_i\in\curB_{J_i}  \bigg\} \\
 \leq 
\sup \bigg\{ \bigg| \nu(A) - \prod_{i=1}^{k} \nu(A_i) \bigg| : A_i\in\curB_{J_i}  \bigg\}  = \alpha_{k}(J_1,\ldots,J_k) ,
\end{multline}
because $\nu(A_{k+1})\le 1$.
Therefore $\alpha_{k+1}(s)\le \alpha_k(s)+\alpha(s)$, which yields the desired conclusion.
\end{proof}

An important example of a (strongly) mixing random field is the case when $\eta$ is a field of i.i.d.\ random elements with law $\rho$, and thus $\alpha(s)=0$ for all $s$. In this case we write $\nu=\nu_\rho$. Note that $\nu_\rho$ is invariant under the $\Gamma$-action on $\Omega$. Given such $\nu_\rho$, consider the product measure $\mu_\vecxi \times \nu_\rho$ on $G_\vecxi\times\Omega$. We denote the push-forward of this measure (restricted to a fundamental domain for the $\Gamma$-action) under the projection map
\begin{equation}
G_\vecxi\times\Omega \to \scrX_\vecxi,\qquad (g,\omega)\mapsto \Gamma (g ,\omega) (\1,\vecxi)
\end{equation}
by $\tmu_{\vecxi,\rho}$. Recall $\mu_\vecxi$ is normalized so that it projects to a probability measure on $X_\vecxi$, which implies $\tmu_{\vecxi,\rho}$ is a probability measure. Note that $\tmu_{\vecxi,\rho}$ is well defined thanks to the $\Gamma$-invariance of $\nu_\rho$.

A random element $\zeta$ in $\scrX_\vecxi$ defines a point process $\Xi=\kappa(\zeta)$ in $\scrM(\RR^d\times Y)$. The Siegel formula for the space of lattices yields:

\begin{lemma}\label{lem:fst}
Let $\vecxi\notin\ZZ^d$. For  $D\in\curB(\RR^d\times Y)$,
\begin{equation}
\EE \Xi D = (\leb\times \rho) D  .
\end{equation}
\end{lemma}

\begin{proof}
It suffices to consider $D=A\times B$ with $A\in\curB(\RR^d)$ and $B\in\curB(Y)$. Since the $\eta(\vecm)$ are independent with law $\rho$, we have $\Xi D= \Theta (A) \rho (B)$. The expectation is $\EE \Xi D= \EE(\Theta A) \rho (B)=\leb (A)\rho (B)$ by Siegel's formula.
\end{proof}

A further special case is when $\{\eta(\vecm), \vecm\in\Z^d\}$ is a collection of independent random variables with $\eta(\vecnull)$ distributed according to $\rho_0$ and $\eta(\vecm)$ distributed according to $\rho$ when $\vecm\ne \vecnull$. In this case we write $\nu=\nu_{\rho_0,\rho}$. Note that $\nu_{\rho_0,\rho}$ is now invariant under the $\Gamma_0$-action on $\Omega$.
Given such $\nu_{\rho_0,\rho}$, consider the product measure $\mu_0 \times \nu_{\rho_0,\rho}$ on $G_0\times\Omega$. We denote the push-forward of this measure (restricted to a fundamental domain for the $\Gamma_0$-action) under the projection map
\begin{equation}
G_0\times\Omega \to \scrX_\vecnull,\qquad (g,\omega)\mapsto \Gamma (g ,\omega)
\end{equation}
by $\tmu_{0,\rho_0,\rho}$. Since $\mu_0$ is normalized so that it projects to a probability measure on $X_0$, also $\tmu_{0,\rho_0,\rho}$ is a probability measure. 
Here $\tmu_{0,\rho_0,\rho}$ is well defined because of the $\Gamma_0$-invariance of $\nu_{\rho_0,\rho}$. 
A random element $\zeta$ in $\scrX_\vecnull$ defines a random product measure $(\varphi,\Xi_0)=\kappa(\zeta)$ in $\scrM(Y)\times\scrM(\RR^d\times Y)$, where $\varphi$ is a point mass and $\Xi_0$ a point process. 

\begin{lemma}\label{lem:scd}
For  $D\in\curB(Y\times\RR^d\times Y)$,
\begin{equation}
\EE (\varphi,\Xi_0) D = (\rho_0\times\leb\times \rho) D .
\end{equation}
\end{lemma}

\begin{proof}
It is sufficient to consider $D=B_0\times A\times B$ with $A\in\curB(\RR^d)$ and $B_0,B\in\curB(Y)$. Since the $\eta(\vecm)$, $\eta(\vecnull)$ are independent with law $\rho$ and $\rho_0$ respectively, we have $\Xi D= \rho_0 (B_0)\Theta_0 (A)  \rho (B)$. The expectation is $\EE \Xi D= \rho_0 (B_0)\EE(\Theta_0 A)  \rho (B)$, and the claim follows from Siegel's formula $\EE\Theta_0 A=\leb A$.
\end{proof}

\section{Spherical averages in the space of marked lattices: convergence on average}\label{sec:main0}

Let $t\in\R$, $M\in\group$, $\vecxi\in\RR^d$, $\omega\in\Omega$, $U\subset\RR^{d-1}$ a bounded set with measure zero boundary (as in Section \ref{sec:lattice}), $\lambda$ an absolutely continuous Borel probability measure on $U$, and $\nu$ a probability measure on $\Omega$ defined by the random field $\eta$.

We define the Borel probability measures $\bP_t^\omega=\bP_t^{(\vecxi,M,\omega,\lambda)}$ and $\bQ_t=\bQ_t^{(\vecxi,M,\nu,\lambda)}$ on $\scrX_\vecxi$ by
\begin{align}\label{eq:measure_def_omega}
\bP_t^\omega f 
&=\int_{\vecu\in U} f(  \Gamma ((\1,\vecxi)M E(\vecu) \Phi^t, \omega )) \lambda(d\vecu) ,
\end{align}
\begin{align}\label{eq:measure_def_omega2}
\bQ_t f 
&= \int_{\omega' \in\Omega} \int_{\vecu\in U} f(   \Gamma ((\1,\vecxi)M E(\vecu) \Phi^t, \omega') ) \lambda(d\vecu) d\nu(\omega'),
\end{align}
for bounded continuous functions $f:\scrX_\vecxi\to\RR$. The principal result of this paper is that $\bP_t^\omega$ converges weakly to $\tmu_{\vecxi,\rho}$ (if $\vecxi\notin\ZZ^d$) or $\tmu_{0,\rho_0,\rho}$ (if $\vecxi\in\ZZ^d$) for $\nu$-almost every $\omega$ (Theorem \ref{th:fixed_omega}). We will first prove this fact for the averaged $\bQ_t$.

\begin{theorem}\label{th:average}
Assume the random field $\eta$ is mixing of all orders with asymptotic distribution $\rho$. Then, for $t\to\infty$,
\begin{equation}
Q_t \toweak 
\begin{cases}
\tmu_{\vecxi,\rho} & (\vecxi\notin\ZZ^d)\\
\tmu_{0,\rho_0,\rho} & (\vecxi\in\ZZ^d),
\end{cases}
\end{equation}
where $\rho_0$ is the law of $\eta(-\vecxi)$.
\end{theorem}

Let $\zeta_t$ be the random element distributed according to $Q_t$, and $\zeta$ according to $\tmu_{\vecxi,\rho}$. Theorem \ref{th:average} $(\vecxi\notin\ZZ^d)$ can then be rephrased as $\zeta_t \todist \zeta$. Similarly, for $\vecxi\in\ZZ^d$, let $\zeta_t$ be the random element distributed according to $Q_t$, and $\zeta$ according to $\tmu_{0,\rho_0,\rho}$. Theorem \ref{th:average} $(\vecxi\in\ZZ^d)$ can then be expressed as well as $\zeta_t \todist \zeta$.
In view of Proposition \ref{topprop1} and the continuous mapping theorem \cite[Theorem 4.27]{Kallenberg02}, this is equivalent to the following convergence in distribution for the random measures $\Xi_t=\kappa(\zeta_t)$ and $\Xi=\kappa(\zeta)$ $(\vecxi\notin\ZZ^d)$, $(\varphi_t,\Xi_{0,t})=\kappa_0(\zeta_t)$ and $(\varphi,\Xi_0)=\kappa_0(\zeta)$ $(\vecxi\in\ZZ^d)$.

\begin{theorem}\label{thm:ppoint22}
Assume $\eta$ is mixing of all orders with asymptotic distribution $\rho$. Then, for $t\to\infty$,
\begin{equation}
\Xi_t \todist \Xi \quad (\vecxi\notin\ZZ^d), \qquad (\varphi_t,\Xi_{0,t}) \todist (\varphi,\Xi_0) \quad (\vecxi\in\ZZ^d).
\end{equation}
\end{theorem}

Theorem \ref{thm:ppoint22} (and hence Theorem \ref{th:average}) follows from the convergence of finite-dimensional distributions by \cite[Theorem 16.16]{Kallenberg02} stated in the following propositions.

\begin{prop}\label{prop:ms_special22}
Let $\vecxi\notin\ZZ^d$, and assume $\eta$ is mixing of all orders with asymptotic distribution $\rho$. Let $n\in\NN$, $D_1,\ldots,D_n\in\curB(\RR^d\times Y)$ bounded with $\Xi\d D_i = 0$ almost surely for all $i$. Then, for $t\to\infty$,
\begin{equation}
(\Xi_t D_1,\ldots,\Xi_t D_n) \todist (\Xi D_1,\ldots,\Xi D_n) .
\end{equation}
\end{prop}

\begin{prop}\label{prop:ms_special22b}
Let $\vecxi\in\ZZ^d$, and assume $\eta$ is mixing of all orders with asymptotic distribution $\rho$. Let $n\in\NN$, $D_1,\ldots,D_n\in\curB(\RR^d\times Y)$ bounded with $\Xi_0\d D_i = 0$ almost surely for all $i$, and $B_0\in\curB(Y)$ with $\varphi\d B_0=0$ almost surely. Then, for $t\to\infty$,
\begin{equation}
(\varphi_t B_0,\Xi_{0,t} D_1,\ldots,\Xi_{0,t} D_n) \todist (\varphi B_0,\Xi_0 D_1,\ldots,\Xi_0 D_n) .
\end{equation}
\end{prop}

\begin{proof}[Proof of Proposition \ref{prop:ms_special22}]
It is sufficient to consider test sets of the form $D_i=A_i\times B_i$ with $A_i\in\curB(\RR^d)$ bounded and  $B_i\in\curB(Y)$ such that $\Xi \d (A_i\times B_i)=0$ almost surely. In view of Lemma \ref{lem:fst} the latter is equivalent to $(\leb\times\rho)\d (A_i\times B_i)=0$. We also assume without loss of generality that $A_i$ are pairwise disjoint. 

Corollary \ref{cor:ms_special}, Lemma \ref{lem:onepoint} and the Chebyshev inequality imply that for every bounded $D_0=A_0\times B_0$ we have
\begin{equation}
\limsup_{t\to\infty} \PP( \Xi_t D_0 \ge 1 ) \le \limsup_{t\to\infty} \PP( \Theta_t A_0 \ge 1 ) \le \PP(\Theta \overline A_0 \le 1) \le \leb \overline A_0 .
\end{equation}
Hence sets $D_0=A_0\times B_0$ where the closure of $A_0$ has small Lebesgue measure have small probability, and we can thus remove such sets from the $D_i$. This explains why, without loss of generality, we may assume from now on that the $A_i$ are convex and that the hyperplane $\{x_1=0\}$ does not intersect the closure $\overline A$ of $A:=\cup_{i=1}^n A_i$. 
Set
\begin{equation}\label{scrL}
\scrL_{t,\vecu}=\Z^d (\1,\vecxi)M E(\vecu)\Phi^t ,
\end{equation}
and write $\vecq=\vecm(\1,\vecxi)M E(\vecu)\Phi^t\in\scrL_{t,\vecu}$ with $\vecm\in\ZZ^d$ uniquely determined by $\vecq$. 
Writing $\vece_1=(1,0,\dots,0)$ for the first standard basis vector in $\R^d$, we have
\begin{equation}\label{same0}
\vecq\cdot \vece_1 =  \vecm (\1,\vecxi) M E(\vecu)\Phi^t \cdot \vece_1
=\e^{-(d-1)t}  (\vecm+\vecxi) M E(\vecu) \cdot \vece_1
\end{equation}
and hence for some constants $c_M>0$, $c_{A,M}>0$ (depending only on $M$ resp.\ $A$ and $M$) 
\begin{equation}\label{cAM}
\big\| \vecm+\vecxi \| \ge c_M \big\| (\vecm+\vecxi) M \| \ge c_M \e^{(d-1)t} \big|\vecq \cdot \vece_1\big| \ge c_{A,M} \e^{(d-1)t} ,
\end{equation}
uniformly for all $\vecq\in A$, $t\ge 0$.

For a small parameter $\eps>0$ to be chosen later write $U=U_1^{(\eps)} \cup U_2^{(\eps)}$, where 
\begin{align}
 U_1^{(\eps)} = \left\{ \vecu\in U : \exists \vecq_1\ne \vecq_2\in A\cap\scrL_{t,\vecu}  \text{ s.t.\ }|(\vecq_1-\vecq_2)\cdot \vece_1| <\eps \right\}
\end{align}
and $U_2^{(\eps)} = U\setminus U_1^{(\eps)}.$ 
The set $U_2^{(\eps)}$ comprises directions corresponding to lattice points $\vecm\in\ZZ^d$ that are $\eps e^{(d-1)t}$-separated. That is, for  
$\vecq_1=\vecm_1(\1,\vecxi)M E(\vecu)\Phi^t$ and $\vecq_2=\vecm_2(\1,\vecxi)M E(\vecu)\Phi^t$ with $\vecu\in U_2^{(\eps)}$,  we have
\begin{equation}\label{same1}
(\vecq_1-\vecq_2)\cdot \vece_1 =  (\vecm_1-\vecm_2)M E(\vecu)\Phi^t \cdot \vece_1
=\e^{-(d-1)t}  (\vecm_1-\vecm_2) M E(\vecu) \cdot \vece_1
\end{equation}
and hence
\begin{equation}\label{cM}
\big\| \vecm_1-\vecm_2 \| \ge c_M \big\| (\vecm_1-\vecm_2) M \| \ge c_M \e^{(d-1)t} \big|(\vecq_1-\vecq_2)\cdot \vece_1\big| \ge c_M \eps \e^{(d-1)t} .
\end{equation}
We will use the higher-order mixing property to show that markings at such points become independent. The set $U_1^{(\eps)}$ includes directions in which there are some lattice points that are close. We will show that the measure of such directions tends to zero as $\eps\to0$.

For non-negative integers $r_1,\ldots, r_n$,
\begin{equation}\label{comb0}
\PP\big(\Xi_t(A_i\times B_i)=r_i \;\forall i \big) =\PP\left(\Xi_t(A_i\times B_i)=r_i \;\forall i \mid \vecu\in U_2^{(\eps)} \right)
+O\left(\lambda(U_1^{(\eps)})\right) .
\end{equation}
We deal with the first term by writing
\begin{multline}
\PP\left(\Xi_t(A_i\times B_i)=r_i \;\forall i \mid \vecu\in U_2^{(\eps)} \right) \\
=\sum_{\mathclap{l_1,\dots,l_n\ge 0}}  \PP\left(\Xi_t(A_i\times B_i)=r_i ,\; \Xi_t(A_i\times Y)=l_i \;\forall i \mid \vecu\in U_2^{(\eps)} \right) .
\end{multline}
Split the summation into terms with $\max_{i} l_i \le L$ and $\max_{i} l_i > L$ for some large $L$. For the latter,
\begin{align}
\sum_{\substack{\mathclap{l_1,\dots,l_n\ge 0}\\
\max_{i} l_i > L}}  & \PP\left(\Xi_t(A_i\times B_i)=r_i ,\; \Xi_t(A_i\times Y)=l_i \;\forall i \mid \vecu\in U_2^{(\eps)} \right) \\ & \le 
\sum_{\substack{\mathclap{l_1,\dots,l_n\ge 0}\\
\max_{i} l_i > L}}   \PP\big( \Xi_t(A_i\times Y)=l_i \;\forall i \big) \\ & \le
\PP\big( \Xi_t(A\times Y)>L  \big)  = \PP\big( \Theta_t(A)>L  \big),
\end{align}
and by Corollary \ref{cor:ms_special} there is $t_0(L,A)$ such that for all $t>t_0(L,A)$,
\begin{equation} \label{eq:large_L}
\PP\big( \Theta_t(A)>L  \big)   \le 1.01 \times \PP\big( \Theta(A)>L  \big) = O( (1+L)^{-d}),
\end{equation}
where the last bound follows from Lemma \ref{lem:decay}. We conclude that, for all $L\ge 1$,
\begin{equation}\label{comb1}
\limsup_{t\to\infty} \sum_{\substack{\mathclap{l_1,\dots,l_n\ge 0}\\
\max_{i} l_i > L}}  \PP\left(\Xi_t(A_i\times B_i)=r_i ,\; \Xi_t(A_i\times Y)=l_i \;\forall i \mid \vecu\in U_2^{(\eps)} \right) = O( (1+L)^{-d}) .
\end{equation}
Let us now turn to the remaining term
\begin{equation}\label{eq:mixing}
\sum_{0\le l_1,\dots,l_n\le L} \PP\left(\Xi_t(A_i\times B_i)=r_i ,\; \Xi_t(A_i\times Y)=l_i \;\forall i \mid \vecu\in U_2^{(\eps)} \right) .
\end{equation}
The only terms which contribute are those with $l_i\ge r_i$. We have
\begin{multline}\label{hym1}
\PP\left(\Xi_t(A_i\times B_i)=r_i ,\; \Xi_t(A_i\times Y)=l_i \;\forall i \mid \vecu\in U_2^{(\eps)} \right) \\
= \frac{1}{\lambda(U_2^{(\eps)})} \int_{U_2^{(\eps)}} \PP\left(\#\{ \vecm\in J_i : \eta(\vecm)\in B_i \}=r_i \;\forall i \mid\vecu  \right) \mathbbm{1}_{\{\#J_1= l_1 \}}\cdots \mathbbm{1}_{\{\#J_n= l_n \}}  d\lambda(\vecu),
\end{multline}
where $J_i=\scrL_{t,\vecu}\cap A_i$.
By the choice of $U_2^{(\eps)}$, all contributing lattice points are $c_M\eps e^{(d-1)t}$-separated, and so, by mixing of all orders,
\begin{multline}
\PP\left(\#\{ \vecm\in J_i : \eta(\vecm)\in B_i \}=r_i \;\forall i \mid \vecu  \right) \\
=\prod_{i=1}^n \bigg( \sum_{S\subset J_i} \bigg(\prod_{\vecm\in S}  \PP\left(\eta(\vecm)\in B_i \right)\times \prod_{\vecm\notin S}  \PP\left(\eta(\vecm)\notin B_i \right) \bigg) \bigg) \\
+ O_L\big(\vartheta_{nL}\big(c_M \eps e^{(d-1)t}\big) \big),
\end{multline}
where the sum is over all subsets $S$ of $J_i$ of cardinality $r_i$. There are $\tbinom{l_i}{r_i}$ such subsets. Again by the choice of $U_2^{(\eps)}$, all contributing lattice points are furthermore at distance at least $c_{A,M} e^{(d-1)t}$ from $\vecxi$. Since $\eta$ has the asymptotic distribution $\rho$, we therefore have  
\begin{multline}
\PP\left(\#\{ \vecm\in J_i : \eta(\vecm)\in B_i \}=r_i \;\forall i \mid \vecu  \right) 
=\prod_{i=1}^n  \tbinom{l_i}{r_i}  p_i^{r_i} (1-p_i)^{l_i-r_i}  \\
+ O_L\big(\vartheta_{nL}\big(c_M \eps e^{(d-1)t}\big)\big) + O_L\big(\beta_{\vecxi}\big( c_{A,M} e^{(d-1)t} \big) \big),
\end{multline}
where $p_i = \rho(B_i)$. Now
\begin{align}
\frac{1}{\lambda(U_2^{(\eps)})} \int_{U_2^{(\eps)}} \mathbbm{1}_{\{\#J_1= l_1 \}}\cdots \mathbbm{1}_{\{\#J_n= l_n \}}  d\lambda(\vecu)
& = \PP\left(\Xi_t(A_i\times Y)=l_i \;\forall i \mid \vecu\in U_2^{(\eps)} \right) \\
& =\PP\left(\Xi_t(A_i\times Y)=l_i \;\forall i \right) +O\left(\lambda(U_1^{(\eps)})\right),
\end{align}
and thus
\begin{multline}
\sum_{l_1=r_1}^L \cdots \sum_{l_n=r_n}^L  \PP\left(\Xi_t(A_i\times Y)=l_i \;\forall i \mid \vecu\in U_2^{(\eps)} \right) \prod_{i=1}^n \tbinom{l_i}{r_i} p_i^{r_i} (1-p_i)^{l_i-r_i}
\\ = \sum_{l_1=r_1}^L \cdots \sum_{l_n=r_n}^L 
\PP\left(\Xi_t(A_i\times Y)=l_i \;\forall i \right) \prod_{i=1}^n \tbinom{l_i}{r_i} p_i^{r_i} (1-p_i)^{l_i-r_i} +O\left(\lambda(U_1^{(\eps)})\right),
\end{multline}
where the implied constants are $l$-independent.
Therefore, using Corollary \ref{cor:ms_special}, 
\begin{multline}
\limsup_{t\to\infty} \bigg| \sum_{0\le l_1,\dots,l_n\le L}  \PP\left(\Xi_t(A_i\times B_i)=r_i ,\; \Xi_t(A_i\times Y)=l_i \;\forall i \mid \vecu\in U_2^{(\eps)} \right) \\  
- \sum_{l_1=r_1}^L \cdots \sum_{l_n=r_n}^L   \PP\left(\Theta(A_i)=l_i \;\forall i \right) \prod_{i=1}^n  \tbinom{l_i}{r_i} p_i^{r_i} (1-p_i)^{l_i-r_i} \bigg| \ll E^{(\eps)} ,
\end{multline}
where
\begin{equation}
E^{(\eps)}:=\limsup_{t\to\infty} \lambda(U_1^{(\eps)}).
\end{equation}

As observed earlier,
\begin{equation}
\sum_{\substack{\mathclap{l_1,\dots,l_n\ge 0}\\
\max_{i} l_i > L}} \PP\left(\Theta(A_i)=l_i \;\forall i \right) = O\left( (1+L)^{-d} \right),
\end{equation}
by Lemma \ref{lem:decay}. This yields
\begin{multline}\label{comb2}
\limsup_{t\to\infty}\bigg| \sum_{0\le l_1,\dots,l_n\le L}  \PP\left(\Xi_t(A_i\times B_i)=r_i ,\; \Xi_t(A_i\times Y)=l_i \;\forall i \mid \vecu\in U_2^{(\eps)} \right) \\  
- \sum_{l_1=r_1}^\infty \cdots \sum_{l_n=r_n}^\infty   \PP\left(\Theta(A_i)=l_i \;\forall i \right)\prod_{i=1}^n  \tbinom{l_i}{r_i} p_i^{r_i} (1-p_i)^{l_i-r_i} \bigg| \ll  (1+L)^{-d} + E^{(\eps)} .
\end{multline}
By the definition of $\Xi$,
\begin{equation}
\sum_{l_1=r_1}^\infty \cdots \sum_{l_n=r_n}^\infty   \PP\left(\Theta(A_i)=l_i \;\forall i \right) \prod_{i=1}^n  \tbinom{l_i}{r_i} p_i^{r_i} (1-p_i)^{l_i-r_i}
=\PP\left(\Xi(A_i\times B_i)=r_i \;\forall i \right)  .
\end{equation}
Combining the estimates \eqref{comb1} and \eqref{comb2} yields for $L\to\infty$,
\begin{equation}\label{comb4}
\limsup_{t\to\infty}
\big| \PP\left(\Xi_t(A_i\times B_i)=r_i \;\forall i \mid \vecu\in U_2^{(\eps)} \right) 
- \PP\left(\Xi(A_i\times B_i)=r_i \;\forall i \right) \big| \ll E^{(\eps)} .
\end{equation}
Here $\eps>0$ is arbitrary. In view of \eqref{comb0} and \eqref{comb4}, what remains to be shown is that $E^{(\eps)}\to 0$ as $\eps\to 0$.
To this end, notice that
\begin{equation}
\begin{split}
\lambda\big(U_1^{(\eps)}\big)  & \le   \sum_{k \in\ZZ} \lambda\big\{ \vecu \in U :\#\big( A
 \cap 
 \big( [k\eps,k\eps +2\eps] \times\R^{d-1}\big)
 \cap
 \Z^d (\1,\vecxi)M E(\vecu) \Phi^t\big) \ge 2 \big\} \\
 & = \sum_{k \in\ZZ} \PP\big( \Theta_t\big( A
 \cap 
 \big( [k\eps,k\eps +2\eps] \times\R^{d-1}\big)\big) \ge 2 \big)  .
 \end{split}
\end{equation}
Since $A$ is bounded, the number of non-zero terms in this sum is $O(1/\eps)$, where the implied constant depends only on $A$ (not on $t$). Taking the limit $t\to\infty$ yields (Corollary \ref{cor:ms_special})
\begin{equation}
E^{(\eps)} \le \sum_{k \in\ZZ} \PP\big( \Theta\big( A
 \cap 
 \big( [k\eps,k\eps +2\eps] \times\R^{d-1}\big)\big) \ge 2 \big)  .
\end{equation}
Because we have assumed that the closure of $A$ does not meet the hyperplane $\{x_1=0\}$, for each $k\ge 0$  the set $A
 \cap 
 \big( [k\eps,k\eps +2\eps] \times\R^{d-1}\big)$
is contained in the cylinder
\begin{equation}
\mathfrak Z(c_1,c_2,C) = \big\{(x_1,\ldots,x_d)\in\RR^d :  c_1 < x_1 < c_2, \|(x_2,\ldots,x_d)\|< C \big\} 
\end{equation}
form some $c_1>0$, $c_2>c_1+2\epsilon$, and $C$ sufficiently large in terms of $A$. (The case of negative $k$ is analogous.)
Therefore, when $d=2$ and $\vecxi\in\Q^2$, we have \cite[Lemma 7.12]{marklof_strombergsson_free_path_length_2010}
\begin{align}
\PP\big( \Theta\big( A
 \cap 
 \big( [k\eps,k\eps +2\eps] \times\R^{d-1}\big)\big) \ge 2 \big)   =O(  \eps^2 \log \eps)  , 
\end{align}
and so $E^{(\eps)} = O( \eps \log \eps)\to 0$ as $\eps\to 0$.
In all other cases we have (use \cite[Lemmas 7.12]{marklof_strombergsson_free_path_length_2010} for $\vecxi\in\QQ^d$, $d\ge 3$, and \cite[Lemmas 7.13]{marklof_strombergsson_free_path_length_2010} for $\vecxi\notin\QQ^d$, $d\ge 2$) 
\begin{align}
\PP\big( \Theta\big( A
 \cap 
 \big( [k\eps,k\eps +2\eps] \times\R^{d-1}\big)\big) \ge 2 \big)  =O(\eps^2)  ,  
\end{align}
that is $E^{(\eps)}  =O  (\eps)\to 0$ as $\eps\to 0$.
\end{proof}
 
 \begin{proof}[Proof of Proposition \ref{prop:ms_special22b}]
The proof is almost the same as that of Proposition \ref{prop:ms_special22}. We have that $\Xi_{0,t}$ is a random point process on $\R^d$ that is jointly measurable with a random variable on $Y$ whose  marginal is $\varphi_t$. We follow the steps of the previous proof  until \eqref{comb0}. 
For $r_i \in\N\cup\{0\}$, $i\ge 1$, we have
\begin{align}\label{eq:xi_integral}
\PP
\big(\varphi_t B_0 =1 &,  \Xi_{0,t}(A_i\times B_i)=r_i \;\forall i \big) 
 \\
&=\PP\left(\varphi_t B_0 = 1, \Xi_{0,t}(A_i\times B_i)=r_i \;\forall i \mid \vecu\in U_2^{(\eps)} \right)
+O\left(\lambda\big(U_1^{(\eps)}\big)\right) .
\end{align}
The next substantive modification is in the application of mixing of order $nL$ in \eqref{eq:mixing}, which becomes
\begin{align}\label{eq:mixing_0}
\sum_{0\le l_1,\dots,l_n\le L}   
& \PP\left(\varphi_t B_0 =1 ,  \Xi_{0,t}(A_i\times B_i)=r_i ,\; \Xi_{0,t}(A_i\times Y)=l_i \;\forall i \mid \vecu\in U_2^{(\eps)} \right) \\  
& =\sum_{l_1=r_1}^L \cdots \sum_{l_n=r_n}^L \PP(\varphi_t B_0 =1 ) \PP\left(\Xi_{0,t}(A_i\times Y)=l_i \;\forall i \mid \vecu\in U_2^{(\eps)} \right) \\
&\times \bigg(\prod_{i=1}^n \tbinom{l_i}{r_i} p_i^{r_i} (1-p_i)^{l_i-r_i}  +  O_L\big(\vartheta_{nL+1}\big(c_M \eps e^{(d-1)t} \big)\big)+O_L\big(\beta_{\vecxi}\big( c_{A,M} e^{(d-1)t} \big) \big)\bigg).
\end{align}
Note that here  $\PP(\varphi_t B_0 = 1) = \PP(\varphi B_0 = 1) = \rho_0(B_0)$. The remainder of the proof runs parallel to that of Proposition \ref{prop:ms_special22}. 
\end{proof}

\section{Spherical averages in the space of marked lattices: almost sure convergence}\label{sec:main}

Let us now turn to the main result of this paper. We say the random field $\eta$ is \emph{slog-mixing} ({\em slog} stands for {\em strongly super-logarithmic}), if 
for every $\delta>0$
\begin{align}\label{eq:mixing_rate0}
 \sum_{t=0}^\infty \alpha(e^{\delta t})<\infty .
\end{align}
This holds for instance when
\begin{equation}
\alpha(s) \le C \, (\log s)^{-1-\eps}
\end{equation}
for all $s\ge 2$, with positive constants $C$, $\eps$. 
 
\begin{maintheorem}\label{th:fixed_omega}
Fix $\vecxi\in\RR^d$ and $M\in G_0$. Assume the random field $\eta$ is  slog-mixing with asymptotic distribution $\rho$. Then there is a set $\Omega_0\subset\Omega$ with $\nu(\Omega_0)=1$, such that for every $\omega\in\Omega_0$ and every a.c.\ Borel probability measure $\lambda$ on $U$, 
\begin{equation}\label{eq:fixed_omega}
P_t^\omega \toweak 
\begin{cases}
\tmu_{\vecxi,\rho} & (\vecxi\notin\ZZ^d)\\
\tmu_{0,\rho_0,\rho} & (\vecxi\in\ZZ^d),
\end{cases}
\end{equation}
as $t\to\infty$, where $\rho_0=\delta_{\omega(-\vecxi)}$. 
\end{maintheorem}

Let $\zeta_t^\omega$ be the random element distributed according to $P_t^\omega$, and $\zeta$ according to $\tmu_{\vecxi,\rho}$. Theorem \ref{th:fixed_omega} $(\vecxi\notin\ZZ^d)$ says that $\zeta_t^\omega \todist \zeta$ for $\nu$-almost every $\omega$. Similarly for $\vecxi\in\ZZ^d$, let $\zeta_t^\omega$ be the random element distributed according to $P_t^\omega$, and $\zeta$ according to $\tmu_{0,\rho_0,\rho}$ where $\rho_0=\delta_{\omega(-\vecxi)}$. Theorem \ref{th:fixed_omega} $(\vecxi\in\ZZ^d)$ can then be expressed as well as $\zeta_t^\omega \todist \zeta^\omega$  for $\nu$-almost-every $\omega$. (The $\omega$-dependence of $\zeta^\omega$ is only through $\rho_0$.)
Again, in view of Proposition \ref{topprop1} and the continuous mapping theorem \cite[Theorem 4.27]{Kallenberg02}, this is equivalent to the following convergence in distribution for the random measures $\Xi_t^\omega=\kappa(\zeta_t^\omega)$ and $\Xi=\kappa(\zeta)$ ($\vecxi\notin\ZZ^d$), $(\varphi_t^\omega,\Xi_{0,t}^\omega)=\kappa_0(\zeta_t^\omega)$, $(\varphi^\omega,\Xi_0)=\kappa_0(\zeta^\omega)$ ($\vecxi\in\ZZ^d$). Note that here $\varphi_t^\omega=\varphi^\omega=\delta_{\omega(-\vecxi)}$. Thus, if $\omega$ is fixed, then $\varphi_t^\omega$ is deterministic and independent of $t$, and we may state the convergence solely for $ \Xi_{0,t}^\omega$ rather than the joint distribution $(\varphi_t^\omega, \Xi_{0,t}^\omega)$ used in the case of random $\omega$ (Proposition \ref{thm:ppoint22}).

\begin{theorem}\label{thm:ppoint22bb}
Under the assumptions of Theorem \ref{th:fixed_omega}, for every $\omega\in\Omega_0$ and every a.c.\ Borel probability measure $\lambda$ on $U$, 
\begin{equation}
\Xi_t^\omega \todist \Xi \quad (\vecxi\notin\ZZ^d), \qquad \Xi_{0,t}^\omega \todist \Xi_0 \quad (\vecxi\notin\ZZ^d).
\end{equation}
\end{theorem}

Again by \cite[Theorem 16.16]{Kallenberg02}, Theorem \ref{thm:ppoint22bb} (and hence Theorem \ref{th:fixed_omega}) follows from the convergence of finite-dimensional distributions:

\begin{prop}\label{prop:fixed_omega}
Under the assumptions of Theorem \ref{th:fixed_omega} with $\vecxi\notin\ZZ^d$, for every $\omega\in\Omega_0$, every a.c.\ Borel probability measure $\lambda$ on $U$, every $n\in\NN$ and all $D_1,\ldots,D_n\in\curB(\RR^d\times Y)$ that are bounded with $\Xi\d D_i = 0$ almost surely for all $i$, 
\begin{equation}
(\Xi_t^\omega D_1,\ldots,\Xi_t^\omega D_n) \todist (\Xi D_1,\ldots,\Xi D_n) .
\end{equation}
\end{prop}

\begin{prop}\label{prop:fixed_omega2}
Under the assumptions of Theorem \ref{th:fixed_omega} with $\vecxi\in\ZZ^d$, for every $\omega\in\Omega_0$, every a.c.\ Borel probability measure $\lambda$ on $U$, every $n\in\NN$ and all $D_1,\ldots,D_n\in\curB(\RR^d\times Y)$ that are bounded with $\Xi_0\d D_i = 0$ almost surely for all $i$,
\begin{equation}
(\Xi_{0,t}^\omega D_1,\ldots,\Xi_{0,t}^\omega D_n) \todist (\Xi_0 D_1,\ldots,\Xi_0 D_n) .
\end{equation}
\end{prop}

The proof of these two propositions will require the following lemma. For each $\zeta>0$, define $\curC_\zeta\subset\curB(\RR^d\times Y)$ as the collection of sets $D=A\times B$ with the following properties:
\begin{enumerate}[(i)]
 \item $A\in\curB(\RR^d)$ is convex and contained in the ball of radius $1/\zeta$ around the origin,
 \item $B\in\curB(Y)$ such that $(\leb\times\rho)\d (A\times B)=0$,
 \item $([-\zeta, \zeta]\times \R^{d-1}) \cap A = \emptyset$.  
\end{enumerate}

\begin{lemma}\label{lem:close_t}
Given $\epsilon>0$, $\zeta<\infty$, there are constants $s_0$, $t_0$ such that for all $t\ge t_0$, $|s|\le s_0$, $\omega\in\Omega$, and every $D\in\curC_\zeta$ (in fact we only require property (i) in the definition of $\curC_\zeta$),
\begin{equation}
\PP\big( \Xi_t^\omega D \neq  \Xi_{t+s}^\omega D \big) <\epsilon , \qquad \PP\big( \Xi_{0,t}^\omega D \neq  \Xi_{0,t+s}^\omega D \big) < \epsilon.
\end{equation}
\end{lemma}

\begin{proof}
We have
\begin{equation}
\begin{split}
\Xi_{t+s}^\omega (A\times B) & = \# \{  (\vecm (\1,\vecxi)M E(\vecu)\Phi^{t+s} , \omega(\vecm)) \in A\times B : \vecm\in\ZZ^d \} \\
& = \# \{  (\vecm (\1,\vecxi)M E(\vecu) \Phi^t , \omega(\vecm)) \in A\Phi^{-s} \times B : \vecm\in\ZZ^d \} \\
& = \Xi_{t}^\omega (A \Phi^{-s} \times B),
\end{split}
\end{equation}
and therefore
\begin{equation}
\PP\big( \Xi_t^\omega D \neq  \Xi_{t+s}^\omega D \big) \le \PP\big( \Theta_t(A \triangle  A \Phi^{-s}) \ge 1 \big) .
\end{equation}
For the latter we have
\begin{equation}
\lim_{t\to\infty}  \PP\big( \Theta_t(A \triangle  A \Phi^{-s}) \ge 1 \big) = \PP\big( \Theta(A \triangle  A \Phi^{-s}) \ge 1 \big) \le \leb (A \triangle  A \Phi^{-s}),
\end{equation}
by Corollary \ref{cor:ms_special} and Lemma \ref{lem:onepoint}. 
The claim follows from the fact that \begin{equation}\lim_{s\to 0}\leb (A \triangle  A \Phi^{-s})=0
\end{equation}
uniformly for any convex $A$ contained in a fixed ball. The proof for $\Xi_{0,t}^\omega$ is identical.
\end{proof}

\begin{proof}[Proof of Proposition \ref{prop:fixed_omega}]

We have
\begin{equation}
\int_\Omega \PP\big( \Xi_t^\omega D_i=r_i \;\forall i \big) d\nu(\omega) 
= \PP\big( \Xi_t D_i=r_i \;\forall i \big) ,
\end{equation}
where $\Xi_t$ is the process considered in Proposition \ref{prop:ms_special22}.
Our first task is to show that
\begin{align}
V^2(t) & := \int_\Omega \big[ \PP\big( \Xi_t^\omega D_i=r_i \;\forall i \big) - \PP\big( \Xi_t D_i=r_i \;\forall i \big) \big]^2 d\nu(\omega)\\
&= \int_\Omega \big[ \PP\big( \Xi_t^\omega D_i=r_i \;\forall i \big) \big]^2 d\nu(\omega)- \PP\big( \Xi_t D_i=r_i \;\forall i \big)^2
\end{align}
decays sufficiently fast for large $t$, uniformly for all a.c.\ Borel probability measures $\lambda$ on $U$ and all $D_1,\ldots,D_n\in\curC_\zeta$ (with $n$ arbitrary but fixed), thus allowing an application of the Borel-Cantelli lemma to establish almost sure convergence.

Let $\Xi_{t,1}^\omega$, $\Xi_{t,2}^\omega$ be two independent copies of $\Xi_t^\omega$. The corresponding rotation parameter $\vecu$ is denoted by $\vecu_1$, $\vecu_2$, respectively, which are independent and distributed according to $\lambda$. Then
\begin{equation}
\int_\Omega \big[ \PP\big( \Xi_t^\omega D_i=r_i \;\forall i \big) \big]^2 d\nu(\omega) = \int_\Omega\PP\big( \Xi_{t,1}^\omega D_i=r_i,\; \Xi_{t,2}^\omega D_i=r_i \;\forall i \big) d\nu(\omega).
\end{equation}
We condition on $\vece_1 E(\vecu_1)^{-1}$, $\pm\vece_1 E(\vecu_2)^{-1}$ being close or not. Let $\theta_0 \le 1$ be small  to be chosen later depending on $t$. 
If $\min_\pm \|\vece_1 E(\vecu_1)^{-1} \pm \vece_1 E(\vecu_2)^{-1}\|\le \theta_0$, then we estimate trivially to get 
\begin{align}
\int_\Omega \PP\big( & \Xi_{t,1}^\omega D_i=r_i,\; \Xi_{t,2}^\omega D_i=r_i \;\forall i \;\big|\; \min_\pm\|\vece_1 E(\vecu_1)^{-1} \pm \vece_1 E(\vecu_2)^{-1}\|\le \theta_0 \big) d\nu(\omega)
\\
&\le  (\lambda\times \lambda ) \big\{ (\vecu_1,\vecu_2) \in U^2 : \min_\pm \|\vece_1 E(\vecu_1)^{-1} \pm \vece_1 E(\vecu_2)^{-1}\| \le \theta_0 \big\} \\
&=  (\lambda\times \lambda ) \left\{ (\vecu_1,\vecu_2) \in U^2 : \|\vece_1 E(\vecu_1)^{-1} - \vece_1 E(\vecu_2)^{-1}\| \le \theta_0 \right\} \label{eq:close_intermediate} ,
\end{align}
for $\theta_0$ sufficiently small, since $\widetilde E(U)$ is contained in a hemisphere (recall the assumptions following \eqref{Etilde}).
Let $\overline M$ be the Lipschitz constant of the inverse of the map $U\to \widetilde E(U)\subset \S^{d-1}_1$, $\vecu \mapsto \vece_1 E(\vecu)^{-1}$. Then, using the fact that $\lambda$ has  density $\lambda'\in L^1(U,d\vecu)$, we  bound \eqref{eq:close_intermediate} by
\begin{align}
 (\lambda\times\lambda)& \left\{ (\vecu_1,\vecu_2)\in U^2: \|\vecu_1 - \vecu_2\| \le \overline M\theta_0 \right\}    
  \\&\le
 (\lambda\times\lambda) \left\{ (\vecu_1,\vecu_2)\in U^2: \|\vecu_1 - \vecu_2\| \le \overline M\theta_0 , \max\{ \lambda'(\vecu_1), \lambda'(\vecu_2)\} \le K\right\} \\
 &\phantom{\ll} 
 +(\lambda\times\lambda) \left\{ (\vecu_1,\vecu_2)\in U^2: \|\vecu_1 - \vecu_2\| \le \overline M\theta_0 , \max\{ \lambda'(\vecu_1), \lambda'(\vecu_2)\} > K \right\} \\
 & \le 
 K^2(\leb\times\leb) \left\{ (\vecu_1,\vecu_2)\in U^2 : \|\vecu_1 - \vecu_2\| \le \overline M\theta_0 \right\}\\
 &\phantom{\ll}
 +\lambda\left\{\vecu\in U : \lambda'(\vecu) >K\right\}\\
 &\ll_U K^2 \overline M^{d-1} \theta_0^{d-1} + \frac1{K} \label{eq:close_regime}
\end{align}
for any $K\ge 1$, where the implied constant depends only on $U$.
If we pick $K = \theta_0^{-(d-1)/3}$, we get the bound $\theta_0^{(d-1)/3}$ for this regime. 
Consider the complementary case, $\|\vece_1 E(\vecu_1)^{-1} \pm \vece_1 E(\vecu_2)^{-1}\| > \theta_0$. Recall that for every $i$, the closure of $A_i$ does not intersect a $\zeta$-neighborhood of the hyperplane $\{x_1=0\}$, and $A_i$ is contained in a ball of radius $1/\zeta$. This implies that the set $A_i \Phi^{-t} E(\vecu)^{-1}$ asymptotically aligns in direction $\pm\vece_1 E(\vecu)^{-1}$, avoiding a $\e^{(d-1)t}$-neighborhood of the origin. More precisely, there is a constant $C_\zeta>0$ such that
\begin{equation}
\| \vecq_1 \Phi^{-t} E(\vecu_1)^{-1} -\vecq_2 \Phi^{-t} E(\vecu_2)^{-1}\| \ge C_\zeta \theta_0 \e^{(d-1)t} 
\end{equation}
for all $\vecq_1\in A_i$, $\vecq_2\in A_j$ and all $i,j$. Hence for $\vecm_1,\vecm_2$ defined by
\begin{equation}
\vecq_1 = (\vecm_1+\vecxi) M E(\vecu_1) \Phi^t,\qquad \vecq_2 = (\vecm_2+\vecxi) M E(\vecu_2) \Phi^t,
\end{equation}
we have
\begin{equation}
\| \vecm_1 -\vecm_2 \| \ge c_M \| (\vecm_1 -\vecm_2)M \| \ge c_M C_\zeta \theta_0 \e^{(d-1)t} .
\end{equation}
This shows that the lattice points $\vecm_1,\vecm_2\in\ZZ^d$ that contribute to $\vecu_1$ and $\vecu_2$ respectively, are at distance at least $c_{M} C_\zeta \theta_0 e^{(d-1)t}$ apart. Thus, by strong mixing,
\begin{align}
\int_\Omega \PP\big( & \Xi_{t,1}^\omega D_i=r_i,\; \Xi_{t,2}^\omega D_i=r_i \;\forall i \;\big|\; \|\vece_1 E(\vecu_1)^{-1} \pm \vece_1 E(\vecu_2)^{-1}\| > \theta_0 \big) d\nu(\omega)
\\ & = \PP\big( \Xi_{t,1} D_i=r_i,\; \Xi_{t,2} D_i=r_i \;\forall i \;\big|\; \|\vece_1 E(\vecu_1)^{-1} \pm \vece_1 E(\vecu_2)^{-1}\| > \theta_0 \big) \\
& + O\big(\alpha(c_M C_\zeta \theta_0 e^{(d-1)t})\big),
\end{align}
where $\Xi_{t,1}$, $\Xi_{t,2}$ are independent copies of $\Xi_t$, and the implicit constant in the error term is independent of the choice of $\lambda$ and of $D_1,\ldots,D_n\in\curC_\zeta$.
Estimate \eqref{eq:close_regime} yields
\begin{align}
\PP\big( & \Xi_{t,1} D_i=r_i,\; \Xi_{t,2} D_i=r_i \;\forall i \;\big|\; \|\vece_1 E(\vecu_1)^{-1} \pm \vece_1 E(\vecu_2)^{-1}\| > \theta_0 \big) \\
& = \PP\big( \Xi_{t,1} D_i=r_i,\; \Xi_{t,2} D_i=r_i \;\forall i \big) +O(\theta_0^{(d-1)/3}) \\
& = \PP\big( \Xi_t D_i=r_i \big)^2 +O(\theta_0^{(d-1)/3}) .
\end{align}

Altogether we therefore have
\begin{align}
\sup_{\lambda,D_1,\ldots,D_n} V^2(t) \ll \alpha(c_M C_\zeta  \theta_0 e^{(d-1)t})+\theta_0^{(d-1)/3},\end{align}
where the supremum is taken over all a.c.\ $\lambda$ and all $D_1,\ldots,D_n\in\curC_\zeta$.
If we choose $\theta_0 = e^{-\gamma t}$ for any $\gamma \in (0,d-1)$, we get 
\begin{align}
 \sup_{\lambda,D_1,\ldots,D_n} \sum_{t\in \delta \N} V^2(t)
   \ll \sum_{t\in \delta \N} \left( \alpha\left(c_M C_\zeta  e^{ (d-1-\gamma) t} \right) + e^{-\gamma(d-1) t/3}\right) < \infty
\end{align}
for every $\delta>0$ by the slog-mixing assumption \eqref{eq:mixing_rate0} and monotonicity of $\alpha$. All of the above estimates are uniform in $\lambda$ and $D_1,\ldots,D_n\in\curC_\zeta$.

From the Borel-Cantelli Lemma we conclude that, for every $\eps>0$, 
\begin{align}
 \nu\left\{\omega\in\Omega :  \sup_{\lambda,D_1,\ldots,D_n} \left\lvert  
 \PP\big( \Xi_{k\delta}^\omega D_i=r_i \;\forall i \big) - \PP\big( \Xi_{k\delta} D_i=r_i \;\forall i \big) \right\rvert > \eps
 \text{ for i.m. }k\in\N\right\} = 0 .
\end{align}
Now choose $\delta>0$ and $k_0$ such that for all $k\ge k_0$, $0\le s < \delta$,
\begin{equation}
\sup_{\lambda,D_1,\ldots,D_n} \big| \PP\big( \Xi_{k\delta}^\omega D_i=r_i \;\forall i \big) - \PP\big( \Xi_{k\delta+s}^\omega D_i=r_i \;\forall i \big) \big|
< \frac{\eps}{2}.
\end{equation}
This is possible in view of Lemma \ref{lem:close_t}, since
\begin{equation}
\big| \PP\big( \Xi_{k\delta}^\omega D_i=r_i \;\forall i \big) - \PP\big( \Xi_{k\delta+s}^\omega D_i=r_i \;\forall i \big) \big|
\le \sum_i \PP\big( \Xi_{k\delta}^\omega D_i\neq \Xi_{k\delta+s}^\omega D_i \big) .
\end{equation}
This shows that (set $t=k\delta+s$)
\begin{equation}\label{lsst}
 \nu\left\{\omega\in\Omega :  \sup_{\lambda,D_1,\ldots,D_n}  \left\lvert  
\PP\big( \Xi_t^\omega D_i=r_i \;\forall i \big) - \PP\big( \Xi_{\delta \lfloor t/\delta \rfloor} D_i=r_i \;\forall i \big) \right\rvert > \frac{\eps}{2}
 \text{ for i.m. } t\in\RR_+\right\} = 0.
\end{equation}
By Proposition \ref{prop:ms_special22}, for every a.c.\ $\lambda$ and all $D_1,\ldots,D_n\in\curC_\zeta$,
\begin{equation}
\lim_{t\to\infty} \PP\big( \Xi_{\delta \lfloor t/\delta \rfloor} D_i=r_i \;\forall i \big) = \PP\big( \Xi D_i=r_i \;\forall i \big) .
\end{equation}
Hence \eqref{lsst} implies that there is a set $\Omega_{\zeta,n}$ of full measure, such that for every $\omega\in\Omega_{\zeta,n}$, all a.c.\ $\lambda$ and all $D_1,\ldots,D_n\in\curC_\zeta$,
\begin{equation} \label{this7}
\lim_{t\to\infty} \PP\big( \Xi_t^\omega D_i=r_i \;\forall i \big)  = \PP\big( \Xi D_i=r_i \;\forall i \big) .
\end{equation}
Corollary \ref{cor:ms_special}, Lemma \ref{lem:onepoint} and Chebyshev's inequality imply that for every $D_0=A_0\times B_0$ we have
\begin{equation}
\limsup_{t\to\infty} \PP( \Xi_t^\omega D_0 \ge 1 ) \le \limsup_{t\to\infty} \PP( \Theta_t A_0 \ge 1 ) \le \PP(\Theta \overline A_0 \le 1) \le \leb \overline A_0
\end{equation}
for all $\omega\in\Omega$. That is, the probability of having at least one point in a small-measure set is small, which shows that \eqref{this7} in fact holds for all sets of the form $D_i=A_i\times B_i$ with $A_i\in\curB(\RR^d)$ bounded and  $B_i\in\curB(Y)$ such that $(\leb\times\rho)\d (A_i\times B_i)=0$, provided
\begin{equation}
\omega\in \Omega_n:=\bigcap_{k=1}^\infty \Omega_{1/k,n}. 
\end{equation}
The convergence in \eqref{this7} holds for all $n$ for a given $\omega$, if 
\begin{equation}
\omega\in \Omega_0:=\bigcap_{n=1}^\infty \Omega_n, 
\end{equation}
which still is a set of full measure. The extension of \eqref{this7} from product sets $A_i\times B_i$ to general sets $D_i$ follows from a standard approximation argument.
\end{proof}

\begin{proof}[Proof of Proposition \ref{prop:fixed_omega2}]
This is identical to the proof of Proposition \ref{prop:fixed_omega}, with Proposition \ref{prop:ms_special22} replaced by Proposition \ref{prop:ms_special22b}. 
\end{proof}

We conclude this section with two corollaries of Theorem \ref{th:fixed_omega}.

\begin{cor}\label{cor:gen}
Under the assumptions of Theorem \ref{th:fixed_omega}, for every $\omega\in\Omega_0$, every a.c.\ $\lambda$, and every bounded continuous $f:\RR^{d-1} \times \scrX_\vecxi \to \RR$, 
\begin{equation}
\lim_{t\to\infty} \int_{U} f( \vecu, \Gamma ((\1,\vecxi)M E(\vecu) \Phi^t, \omega )) \lambda(d\vecu) 
=  \int_{U\times\scrX_\vecxi} f \, d\lambda \times \begin{cases}
d\tmu_{\vecxi,\rho} & (\vecxi\notin\ZZ^d)\\
d\tmu_{0,\rho_0,\rho} & (\vecxi\in\ZZ^d),
\end{cases}
\end{equation}
where $\rho_0=\delta_{\omega(-\vecxi)}$.
\end{cor}

\begin{proof}
This follows from Theorem \ref{th:fixed_omega} by the same argument as in the proof of Theorem 5.3 in \cite{marklof_strombergsson_free_path_length_2010}.
\end{proof}

Let us assume that there is a continuous map $\varphi:U\times \Omega\to \Omega$.
Then the following is an immediate consequence of Corollary \ref{cor:gen}.

\begin{cor}\label{cor:gen1}
Under the assumptions of Theorem \ref{th:fixed_omega}, for every $\omega\in\Omega_0$, every a.c.\ $\lambda$, and every bounded continuous $f:\RR^{d-1} \times \scrX_\vecxi \to \RR$, 
\begin{multline}
\lim_{t\to\infty}  \int_{U} f(\Gamma ((\1,\vecxi)M E(\vecu) \Phi^t, \varphi(\vecu, \omega ))) \lambda(d\vecu) 
\\ =  \int_{U\times\scrX_\vecxi} f(\Gamma(g,\varphi(\vecu,\omega'))) \, d\lambda(\vecu) \times \begin{cases}
d\tmu_{\vecxi,\rho}(g,\omega') & (\vecxi\notin\ZZ^d)\\
d\tmu_{0,\rho_0,\rho}(g,\omega') & (\vecxi\in\ZZ^d),
\end{cases}
\end{multline}
where $\rho_0=\delta_{\omega(-\vecxi)}$.
\end{cor}

\begin{proof}
Apply Corollary \ref{cor:gen} with the test functions $\tilde f$ defined by
\begin{equation}
\tilde f(\vecu, \Gamma (g,\omega)) = f(\Gamma (g, \varphi(\vecu, \omega ))) .
\end{equation}
\end{proof}

\section{Random defects \label{sec:defect}}

Spherical averages were used in \cite{marklof_strombergsson_free_path_length_2010} and \cite{marklof_strombergsson_free_path_length_2014} to establish the limit distribution for the free path length in crystals and quasicrystals, respectively. The plan for the remainder of this paper is to explain how spherical averages on marked lattices can be exploited to yield the path length distribution for crystals with random defects. The idea is to start with a perfect crystal, whose scatterers are located at the vertices of an affine Euclidean lattice $\scrL=\Z^d (\1,\vecxi)M$, and then remove or shift each lattice point with a given probability. This can be encoded by a marking of $\scrL$ as follows. The set of marks is $Y=\{0,1\} \times \RR^d$, where the first coordinate describes the absence  or presence of a lattice point, and the second its relative shift measured in units of $r=\e^{-t}$. The corresponding marking is denoted by $\omega=(a,\vecz)$ with $a:\ZZ^d \to \{0,1\}$ and $\vecz:\ZZ^d \to \RR^d$.
The {\em defect affine lattice} is thus
\begin{equation}
\{ (\vecm+\vecxi)M + r \vecz(\vecm) : \vecm\in\ZZ^d \text{\ s.t.\ } a(\vecm)=1 \}.
\end{equation}
In the case when $\vecxi\in\ZZ^d$, it is natural to shift the above point set by $-r\vecz(-\vecxi)$ so that the shifted set contains the origin. To unify notation, let us therefore define the field $\vecz_\vecxi$ by
\begin{equation}
\vecz_\vecxi(\vecm)=
\begin{cases}
\vecz(\vecm)-\vecz(-\vecxi) & (\vecxi\in\ZZ^d) \\
\vecz(\vecm) & (\vecxi\notin\ZZ^d).
\end{cases}
\end{equation}
In fact, for our application to the Lorentz gas, it will be convenient to shift the point set by a more general vector $r\vecbeta$, where $\vecbeta$ is a fixed bounded continuous function $U\to\RR^d$; we denote the shifted set (for all $\vecxi\in\RR^d$) by
\begin{equation}
\widetilde\scrP_{r,\vecu}=\{ (\vecm+\vecxi)M + r [\vecz_\vecxi(\vecm)-\vecbeta(\vecu)] : \vecm\in\ZZ^d \text{\ s.t.\ } a(\vecm)=1 \} .
\end{equation}
As in the case of lattices \eqref{scrL}, we are interested in the rotated-stretched point set $\scrP_{t,\vecu}=\widetilde\scrP_{r,\vecu} E(\vecu)\Phi^t$, which reads explicitly (for $r=\e^{-t}$)
\begin{multline}
\scrP_{t,\vecu}=\{ (\vecm+\vecxi)M E(\vecu)\Phi^t + ([\vecz_\vecxi(\vecm)-\vecbeta(\vecu)]E(\vecu))_\perp \\ +\e^{-d t} (\vece_1\cdot [\vecz_\vecxi(\vecm)-\vecbeta(\vecu)]E(\vecu)) \vece_1 : \vecm\in\ZZ^d \text{\ s.t.\ } a(\vecm)=1 \},
\end{multline}
where $(\,\cdot\,)_\perp$ is the orthogonal projection onto the hyperplane perpendicular to $\vece_1$. 

We map the marked affine lattice (viewed as an element in $\scrX$) to a defect lattice (viewed as an element in $\scrM(\RR^d)$) by
\begin{equation}
\sigma \colon \scrX \to \scrM(\RR^d), \qquad \Gamma(g,\omega) \mapsto
\sum_{\vecy\in\ZZ^d g} a(\vecy g^{-1})\, \delta_{\vecy+\vecz(\vecy g^{-1})} \qquad (\vecxi\notin\ZZ^d),
\end{equation}
\begin{equation}
\sigma_0 \colon \scrX_\vecnull \to \scrM(\RR^d), \qquad \Gamma(g,\omega) \mapsto
\sum_{\vecy\in\ZZ^d g\setminus\{\vecnull\}} a(\vecy g^{-1})\, \delta_{\vecy+\vecz(\vecy g^{-1})}
\qquad (\vecxi\in\ZZ^d), 
\end{equation}
where $\omega=(a,\vecz)$.
The motivation for this definition is as follows. Define the family of maps $J_t:\Omega\to\Omega$ by \begin{equation}
J_t(\omega) = J_t(a,\vecz)= \big(a,\vecz_\perp +\e^{-d t} (\vece_1\cdot \vecz)\vece_1\big) 
\end{equation}
and (for later use)
\begin{equation}
J_\infty(\omega) = J_\infty(a,\vecz)= (a,\vecz_\perp) .
\end{equation}
Then
\begin{equation}\label{sigma}
\sigma\big( \Gamma \big((\1,\vecxi)M E(\vecu)\Phi^t , J_t(a,[\vecz_\vecxi-\vecbeta(\vecu)] E(\vecu) \big) \big)= \sum_{\vecy\in \scrP_{t,\vecu}} \delta_\vecy 
\end{equation}
and, for $\vecxi\in\ZZ^d$, 
\begin{equation}\label{sigma0}
\sigma_0 \big( \Gamma \big( (\1,\vecxi)M E(\vecu)\Phi^t , J_t(a,[\vecz_\vecxi-\vecbeta(\vecu)] E(\vecu) \big) \big) = \sum_{\vecy\in \scrP_{t,\vecu}\setminus\{\vecnull\}} \delta_\vecy .
\end{equation}
 
We will first discuss the relevant spherical averages in $\scrX_\vecxi$, and then show they map to the above point processes. 
 
\begin{theorem}\label{thm:defect}
Assume $\eta$ is  slog-mixing with asymptotic distribution $\rho$, and $\rho$ has compact support. Fix $M\in G_0$ and $\vecxi\in\RR^d$. Then there exists a set $\Omega_0\subset\Omega$ with $\nu(\Omega_0)=1$, such that for every $\omega=(a,\vecz)\in\Omega_0$, every a.c.\ $\lambda$ and every bounded continuous $f:\RR^{d-1} \times \scrX_\vecxi \to \RR$, 
\begin{multline}\label{eq:defect}
\lim_{t\to\infty} \int_{U} f\big(\Gamma \big((\1,\vecxi)M E(\vecu) \Phi^t, J_t(a,[\vecz_\vecxi-\vecbeta(\vecu)])\big)\big) \lambda(d\vecu) 
\\ =  \int_{\scrX_\vecxi} f\big(\Gamma\big(g,J_\infty(a',[\vecz_\vecxi'-\vecbeta(\vecu)] E(\vecu))\big)\big) \, d\lambda(\vecu) \times \begin{cases}
d\tmu_{\vecxi,\rho}(g,(a',\vecz')) & (\vecxi\notin\ZZ^d)\\
d\tmu_{0,\rho_0,\rho}(g,(a',\vecz')) & (\vecxi\in\ZZ^d),
\end{cases}
\end{multline}
where $\rho_0=\delta_{\omega(-\vecxi)}$.
\end{theorem}
 
\begin{proof}
Since $\rho$ has compact support and $\vecbeta$ is bounded,
\begin{equation}
\sup_{u\in U} \sup_{\vecm\in\ZZ^d} | \vece_1\cdot [\vecz_\vecxi(\vecm)-\vecbeta(\vecu)] E(\vecu ) | <\infty,
\end{equation}
and hence $J_t(a,[\vecz_\vecxi-\vecbeta(\vecu)] E(\vecu))(\vecm)\to J_\infty(a,[\vecz_\vecxi-\vecbeta(\vecu)] E(\vecu))(\vecm)$ uniformly in $\vecu\in U$, $\vecm\in\ZZ^d$. 
The claim now follows from Corollary \ref{cor:gen1}.
\end{proof} 
 
The following is the key to translate the above convergence into the setting of point processes.
 
\begin{lemma}\label{lem:sigma}
The maps $\sigma$ and $\sigma_0$ are continuous.
\end{lemma}

\begin{proof}
The proof is similar to that of Proposition \ref{topprop1}; we sketch it in the case of $\sigma$. 

We need to show that $x_j\to x \in\scrX$ implies that, for every $f\in C_c(\R^d)$,
\begin{align}\label{eq:sigma_cont}
 \sum_{\vecy \in \Z^d g_j} a_j(\vecy g_j^{-1}) f(\vecy + \vecz_j(\vecy g_j^{-1})) \to \sum_{\vecy \in \Z^d g} a(\vecy g^{-1}) f(\vecy + \vecz(\vecy g^{-1})).
\end{align}
Since $f$ is of compact support, the sums above are finite, and we can rewrite the left hand side as 
\begin{align}
 \sum_{\vecy \in \Z^d g_j} a_j(\vecy g_j^{-1}) f(\vecy + \vecz_j(\vecy g_j^{-1}))
 =
 \sum_{\vecm \in \Z^d} a_j(\vecm) f(\vecm g_j + \vecz_j(\vecm)),
\end{align}
which is another finite sum. In particular, for all $\vecm$ in the support, we have that $a_j(\vecm) = a(\vecm)$ and $|\vecz_j(\vecm) - \vecz(\vecm)|<\eps $ once $j\ge j_0$. The statement \eqref{eq:sigma_cont} now follows from continuity of $f$. 
\end{proof} 
 
For $\vecu$ randomly distributed according to $\lambda$, we define the random point processes 
\begin{equation}\label{sigma00}
\widetilde\Xi_t^\omega= \sum_{\vecy\in \scrP_{t,\vecu}} \delta_\vecy , \qquad
\widetilde\Xi_{0,t}^\omega = \sum_{\vecy\in \scrP_{t,\vecu}\setminus\{\vecnull\}} \delta_\vecy 
\end{equation}
for $\vecxi\notin\ZZ^d$ and $\vecxi\in\ZZ^d$, respectively.
If $\vecxi\notin\ZZ^d$, we furthermore set 
\begin{equation}
\widetilde\Xi=\sigma\big(\Gamma\big(g,J_\infty(a,\vecz E(\vecu))\big)\big)
\end{equation}
with $(g,(a,\vecz))$ distributed according to $\tmu_{\vecxi,\rho}$ and $\vecu$ distributed according to $\lambda$. That is, $\widetilde\Xi$ is a random affine lattice $\ZZ^d g$ distributed according to $\mu$, where each lattice point is removed, or shifted in the hyperplane $V_\perp=\{0\}\times\RR^{d-1}$, according to the push-forward of the probability measure $\lambda\times\rho$ on $U\times\{0,1\}\times\RR^d$ under the map $(\vecu,a,\vecz)\mapsto (a,(\vecz E(\vecu))_\perp)$. If $\rho$ is rotation-invariant, then this measure is independent of $\lambda$.
In the case $\vecxi\in\ZZ^d$, we put 
\begin{equation}
\widetilde\Xi_0=\sigma_0\big(\Gamma\big(g,J_\infty(a,[\vecz_\vecxi-\vecbeta(\vecu)] E(\vecu))\big)\big)
\end{equation}
with $(g,(a,\vecz))$ distributed according to $\tmu_{0,\rho_0,\rho}$ and $\vecu$ distributed according to $\lambda$.
This means that $\widetilde\Xi_0$ is a random lattice $\ZZ^d g\setminus\{\vecnull\}$ distributed according to $\mu_0$, where each lattice point is removed, or shifted in the hyperplane $V_\perp=\{0\}\times\RR^{d-1}$, according to the push-forward of the probability measure $\lambda\times\rho$ on $U\times\{0,1\}\times\RR^d$ under the map 
\begin{equation}
(\vecu,a,\vecz)\mapsto (a,([\vecz-\vecz(-\vecxi)-\vecbeta(\vecu)]E(\vecu))_\perp).
\end{equation}
This measure depends on $\lambda$ even if $\rho$ is rotation invariant.

Theorem \ref{thm:defect} implies via Lemma \ref{lem:sigma} and the continuous mapping theorem the following convergence in distribution.

\begin{cor}\label{thm:ppoint22bbcc}
Under the conditions of Theorem \ref{thm:defect}, for every $\omega\in\Omega_0$ and every a.c.\ $\lambda$,
\begin{equation}
\widetilde\Xi_t^\omega \todist \widetilde\Xi \quad (\vecxi\notin\ZZ^d), \qquad \widetilde\Xi_{0,t}^\omega \todist \widetilde\Xi_0 \quad (\vecxi\in\ZZ^d).
\end{equation}
\end{cor}

The following Siegel-Veech type formula allows us to simplify the assumptions on the test sets for the finite-dimensional distribution. Set
\begin{equation}
\overline\rho:= \rho(1,\RR^d) .
\end{equation}

\begin{lemma}\label{lem:tilde}
For  $A\in\curB(\RR^d)$,
\begin{equation}
\EE \widetilde\Xi A = \overline\rho\leb A  \quad (\vecxi\notin\ZZ^d), \qquad \EE \widetilde\Xi_0 A = \overline\rho\leb A \quad (\vecxi\in\ZZ^d).
\end{equation}
\end{lemma}

\begin{proof}
For $f\in C_c(\RR^d)$,
\begin{equation}
\EE \widetilde\Xi f = \int_{\scrX_\vecxi} \sum_{\vecm\in\ZZ^d} a(\vecm) f(\vecm g+([\vecz(\vecm)-\vecbeta(\vecu)]E(\vecu))_\perp))  d\lambda(\vecu)\, d\tmu_{\vecxi,\rho}(g,(a,\vecz)) .
\end{equation}
By Lemma \ref{lem:fst}, we have then
\begin{equation}
\EE \widetilde\Xi f =  \int_{\RR^d \times \RR^d\times U} f(\vecx + ([\vecy -\vecbeta(\vecu)]E(\vecu))_\perp))  \leb(d\vecx) \rho(1,d\vecy) d\lambda(\vecu) ,
\end{equation}
which, after translating $\vecx$ by $-([\vecy -\vecbeta(\vecu)]E(\vecu))_\perp)$, yields
\begin{equation}
\EE \widetilde\Xi f =  \int_{\RR^d \times \RR^d\times U}  f(\vecx)  \leb(d\vecx) \rho(1,d\vecy) d\lambda(\vecu) =\overline\rho\leb f.
\end{equation}
The proof for $\widetilde\Xi_0$ is identical.
\end{proof}

The following is a direct consequence of Corollary \ref{thm:ppoint22bbcc} and Lemma \ref{lem:tilde}.

\begin{cor}\label{prop:fixed_omega1234}
Assume the conditions of Theorem \ref{thm:defect}. Then, for every $\omega\in\Omega_0$, every a.c.\ $\lambda$, every $n\in\NN$ and all $A_1,\ldots,A_n\in\curB(\RR^d)$ that are bounded with $\leb\d A_i = 0$ for all $i$,
\begin{equation}
(\widetilde\Xi_t^\omega A_1,\ldots,\widetilde\Xi_t^\omega A_n) \todist (\widetilde\Xi A_1,\ldots,\widetilde\Xi A_n) \quad (\vecxi\notin\ZZ^d)
\end{equation} 
and
\begin{equation}
(\widetilde\Xi_{0,t}^\omega A_1,\ldots,\widetilde\Xi_{0,t}^\omega A_n) \todist (\widetilde\Xi_0 A_1,\ldots,\widetilde\Xi_0 A_n) \quad (\vecxi\in\ZZ^d).
\end{equation}
\end{cor}

\begin{proof}
Lemma \ref{lem:tilde} implies that $\leb\d A_i = 0$ implies $\widetilde\Xi\d A_i = 0$ almost surely and $\widetilde\Xi_0\d A_i = 0$ almost surely, and the claim follows from \cite[Theorem 16.16]{Kallenberg02}.
\end{proof}

\section{Free path lengths in the Lorentz gas \label{sec:lorentz}}

For a given point set $\scrP\subset\RR^d$, center an open ball $\scrB_r^d+\vecy$ of radius $r$ at each of the points $\vecy$ in $\scrP$. The Lorentz gas describes the dynamics of point particle in this array of balls, where the particle moves with unit velocity until it hits a ball, where it is scattered according to a given scattering map. The configuration space for the dynamics is thus $\scrK_r=\RR^d\setminus(\scrB_r^d+\scrP)$. Given the initial position $\vecq\in\scrK_r$ and velocity $\vecv\in\S_1^{d-1}$, the free path length is defined as the travel distance until the next collision, 
\begin{equation}
\tau(\vecq,\vecv;r):=\inf\{ t>0 : \vecq+t\vecv\notin \scrK_r \}.
\end{equation}
The distribution of the free path length is well understood for random \cite{boldrighini_bunimovich_sinai_1983}, periodic \cite{boca_zaharescu_2007, caglioti_golse_2010,marklof_strombergsson_free_path_length_2010} and quasiperiodic \cite{marklof_strombergsson_free_path_length_2014,power-law_2014} scatterer configurations. We will here consider the periodic Lorentz gas with random defects introduced in the previous section, where the scatterers are placed at the defect lattice
\begin{equation}
\scrP_r=\{ \vecm M + r \vecz(\vecm) : \vecm\in\ZZ^d \text{\ s.t.\ } a(\vecm)=1 \}.
\end{equation}
Note that the papers \cite{cagliogi_pulvirenti_ricci_2000,ricci_wennberg_2004} discuss the convergence of a defect periodic Lorentz gas to a random flight process governed by the linear Boltzmann equation {\em in the limit when the removal probability of a scatterer tends to one}. In this case the free path length distribution is exponential, whereas for a {\em fixed} removal probability $<1$ the path length distribution has a power-law tail; cf.\ \eqref{eq:plt}.

As in  \cite{marklof_strombergsson_free_path_length_2010}, we will consider more general initial conditions, which for instance permit us to launch a particle from the boundary of a scatterer (which is moving as $r\to 0$).
Let $\vecbeta:\S_1^{d-1}\to\RR^d$ be a continuous
function, and consider the initial condition $\vecq+r\vecbeta(\vecv)$. If $\vecq\in\ZZ^d M$ and the initial condition is thus very near (within distance $O(r)$) to a scatterer, we will avoid initial conditions inside the scatterer, or those that immediately hit the scatterer, by assuming that $\vecbeta$ is such that the ray $\vecbeta(\vecv)+\R_{>0}\vecv$ lies completely outside $\scrB_1^d$, for each $\vecv\in\S_1^{d-1}$. 

The following theorem proves the existence of the free path length distribution for (a) random initial data $(\vecq+r\vecbeta(\vecv),\vecv)$ for $\vecq\notin\ZZ^d M$ fixed, and $\vecv$ random with law $\lambda$, and (b) random initial data $(\vecq+r\vecbeta(\vecv)+r\vecz(\vecq M^{-1}),\vecv)$ for $\vecq\in\ZZ^d M$ fixed, and $\vecv$ random with law $\lambda$.

\begin{theorem}
Assume $\eta$ is  slog-mixing with asymptotic distribution $\rho$, and $\rho$ has compact support. Fix $M\in G_0$ and $\vecq\in\RR^d$. 
There exist continuous, non-increasing functions $F_s\colon \RR_{\ge 0} \to \R$ with $F_s(0)=1$ ($s\in\ZZ_{\ge 0}$) and a set $\Omega_0\subset\Omega$ with $\nu(\Omega_0)=1$, such that the following hold for every $T\ge 0$, every $(a,\vecz)\in \Omega_0$ and every absolutely continuous Borel probability measure $\lambda$ on $\S_1^{d-1}$:
\begin{enumerate}[(i)]
\item If $\vecq\in\RR^d\setminus\QQ^d M$ and $T\ge 0$, then
\begin{equation}
\lim_{r\to 0} \lambda\{ \vecv\in\S_1^{d-1} : r^{d-1} \tau(\vecq+r\vecbeta(\vecv),\vecv;r) \ge T \} = F_0(T) .
\end{equation}
\item If $\vecq=s^{-1}\vecm M$ with $s\in\ZZ_{\ge 2}$, $\vecm\in\ZZ^d$, $\gcd(\vecm,s)=1$, and $T\ge 0$, then
\begin{equation}
\lim_{r\to 0} \lambda\{ \vecv\in\S_1^{d-1} : r^{d-1} \tau(\vecq+r\vecbeta(\vecv),\vecv;r) \ge T \} = F_s(T) .
\end{equation}
\item If $\vecq\in\ZZ^d M$ and $T\ge 0$, then
\begin{equation}
\lim_{r\to 0} \lambda\{ \vecv\in\S_1^{d-1} : r^{d-1} \tau(\vecq+r\vecbeta(\vecv)+r\vecz(\vecq M^{-1}),\vecv;r) \ge T \} = F_1(T) .
\end{equation}
\end{enumerate}
\end{theorem}

\begin{proof}
The proof follows from Corollary \ref{prop:fixed_omega1234} in the case of one-dimensional distributions ($n=1$) by the same arguments as in \cite{marklof_strombergsson_free_path_length_2010}. This proves the existence of the limits with
\begin{align}
F_0(T) &= \PP\big(\widetilde\Xi Z(T,1) =0\big) \quad \text{with $\vecxi\notin\QQ^d$,}
\\
F_s(T)& = \PP\big(\widetilde\Xi Z(T,1) =0\big) \quad \text{with $\vecxi=s^{-1}\vecm$,}
\\
F_1(T)& = \PP\big(\widetilde\Xi_0 Z(T,1) =0\big) \quad \text{with $\vecxi\in\ZZ^d$,}
\end{align}
where
\begin{equation}\label{zyl}
	Z(T,R) :=\big\{(x_1,\ldots,x_d)\in\RR^d : 0 < x_1 < T, \|(x_2,\ldots,x_d)\|< R \big\} .
\end{equation}
Note that the limit process $\widetilde\Xi$ is independent of the choice of $\vecxi$ when $\vecxi\notin\QQ^d$, and only depends on the denominator of $\vecxi$ when $\vecxi\in\QQ^d\setminus\ZZ^d$; cf.\ \cite{marklof_strombergsson_free_path_length_2010} for a detailed discussion. Furthermore $\widetilde\Xi_0$ is independent of $\vecxi\in\ZZ^d$.
\end{proof}

Let $r_{\max}$ be the infimum over the radii of balls centered at the origin that contain the support of $\rho(1,\,\cdot\,)$ (which we have assumed to be compact). Then the maximal distance between a point in the random affine lattice $\Theta$ and its displacement in $\widetilde\Xi$ is $r_{\max}$.
Denote by $\overline F_s$ the corresponding path length distribution 
\begin{equation}
\overline F_s(T) = \PP\big(\Theta Z(T,1) =0\big) \quad (s\neq 1), \qquad \overline F_1(T) = \PP\big(\Theta_0 Z(T,1) =0\big).
\end{equation}

\begin{lemma}
For $T\ge 0$,
\begin{equation}
F_s(T) \ge \overline F_s\big((1+r_{\max})^{d-1}T\big) .
\end{equation}
\end{lemma}

\begin{proof}
We have
\begin{equation}
\PP\big(\widetilde\Xi Z(T,1) =0\big) \ge \PP\big(\Theta Z(T,1+r_{\max}) =0\big) 
= \PP\big(\Theta Z((1+r_{\max})^{d-1}T,1) =0\big) ,
\end{equation}
where the last equality follows from the $G_0$-invariance of $\Theta$.
\end{proof}

This lemma allows us to obtain lower bounds for the tails of $F_s(T)$ in terms of the free path length asymptotics derived in \cite{marklof_strombergsson_periodic_2011}. In particular, Theorem 1.13 in that paper implies the power-law lower bound
\begin{equation}\label{eq:plt}
F_0(T) \ge \frac{\pi^{\frac{d-1}2}(1+r_{\max})^{1-d}}{2^{d}d\, \Gamma(\frac {d+3}2)\,\zeta(d)} \; T^{-1}
+O\bigl(T^{-1-\frac 2d}\bigr) .
\end{equation}
Note that this bound becomes ineffective in the limit of large $r_{\max}$. The bound is also consistent with the exponential distribution in the limit of removal probability $\to 1$ discussed in \cite{cagliogi_pulvirenti_ricci_2000,ricci_wennberg_2004}, if the free path length is measured in units of the mean free path length, which diverges as the removal probability tends to one.



\bibliographystyle{plain}
\bibliography{bibliography}

\end{document}